\documentclass[12pt,a4paper]{amsart}
\usepackage[T1]{fontenc}
\usepackage{lmodern}
\usepackage{amsmath,amssymb,amsthm,mathrsfs}
\usepackage{enumitem,microtype,url,needspace}
\usepackage[a4paper,left=25mm,right=25mm,top=27mm,bottom=27mm,headheight=14pt]{geometry}
\usepackage[colorlinks=true,linkcolor=blue,citecolor=blue,urlcolor=blue]{hyperref}
\theoremstyle{plain}
\newtheorem{theorem}{Theorem}[section]
\newtheorem{proposition}[theorem]{Proposition}
\newtheorem{lemma}[theorem]{Lemma}
\newtheorem{corollary}[theorem]{Corollary}
\newtheorem{maintheorem}{Theorem}

\theoremstyle{definition}
\newtheorem{definition}[theorem]{Definition}
\newtheorem{example}[theorem]{Example}
\theoremstyle{remark}
\usepackage{advdate}
\newtheorem{remark}[theorem]{Remark}
\newtheorem{problem}[theorem]{Problem}
\DeclareMathOperator{\FBL}{FBL}
\DeclareMathOperator{\SB}{SB}
\DeclareMathOperator{\Subl}{Subl}
\DeclareMathOperator{\Ideal}{Ideal}
\DeclareMathOperator{\dist}{dist}

\newcommand{\R}{\mathbb{R}}
\newcommand{\N}{\mathbb{N}}
\newcommand{\bbQ}{\mathbb{Q}}
\title{Polish spaces of separable Banach lattices}
\author[M.~Niwi\'nski]{Mariusz Niwi\'nski}
\address{Doctoral School of Exact and Natural Sciences, Faculty of Mathematics
and Computer Science, Institute of Mathematics, Jagiellonian University,
{\L}ojasiewicza 6, 30-348 Krak\'ow, Poland}
\email{mariusz.niwinski@doctoral.uj.edu.pl}
\date{\AdvanceDate[-19]\today}
\subjclass[2020]{46B42, 03E15, 46B04, 54B20}
\keywords{Banach lattice, free Banach lattice, Polish space, Wijsman topology,
Fremlin tensor product, uniform monotonicity}
\begin{document}
\begin{abstract}
We compare two Polish presentations of separable real Banach lattices:
closed sublattices of a universal lattice and closed ideals of
$\FBL[\ell_1]$. We construct translations preserving lattice isometry in
both directions, continuous and injective from sublattices to ideals and
Borel in the reverse direction. Natural generation and Fremlin tensor
operations are $\boldsymbol\Sigma^0_2$-measurable with $G_\delta$ graphs.
Isometric lattice isomorphism in either universal coding is Borel
bireducible with a complete orbit equivalence relation of Polish-group
actions. Uniform monotonicity and order uniform smoothness are
$\boldsymbol\Pi^0_3$-complete in both codings, whereas strict monotonicity,
order smoothness, and order continuity are coanalytic. Abstract $L_p$
and AM-spaces form closed classes of sublattice codes. We also give
explicit Borel bounds for atom counts, compact order intervals, and the
lattice-isometry classes of $L_p[0,1]$, $\ell_p$, and $c_0$.
\end{abstract}
\maketitle

\section{Introduction}

Descriptive set theory separates two questions about a class of mathematical
structures. One may ask how difficult it is to classify its members up to
isomorphism, or how difficult it is to recognise a particular property of a
single member. The second question requires a useful space of presentations.
For separable Banach spaces, the Effros--Borel coding and its Polish refinements
have made it possible to study geometric properties and individual isometry
classes by descriptive methods. We refer to Bossard~\cite{Boss2002Cod},
Dodos~\cite{Dod2010Ban}, and Godefroy--Saint-Raymond~\cite{God2018Cod} for the
foundations. More recently, C\'uth, Dole\v zal, Doucha, and
Kurka~\cite{CDDK2022Polish,CDDK2024Complexity} developed norm and pseudonorm parameter spaces
and obtained sharp complexity results for classical Banach spaces. In the setting of Banach lattices, the first coding, by embeddings, was introduced by Tursi in her doctoral dissertation \cite{Tursi2021}.

A Banach lattice carries more information than its underlying Banach space:
the modulus, positive cone, and order ideals must be retained by the code.
There are two natural ways to do this. The aforementioned universal embedding theorem of
Leung, Li, Oikhberg, and Tursi~\cite{LLOT2019Uni} represents every separable
Banach lattice as a closed sublattice of
\[
  \mathcal C=C(\Delta;L_1(0,1)),\qquad \Delta=\{0,1\}^{\N}.
\]
The theory of free Banach lattices instead represents it as a quotient of
$P=\FBL[\ell_1]$ by a closed order ideal. The free constructions of
de~Pagter--Wickstead~\cite{de2015free} and
Avil\'es--Rodr\'\i guez--Tradacete~\cite{aviles2018free}, together with the
recent systematic treatment~\cite{OTTT2026Free,OTTT2026Corr}, supply the relevant
background.

The quotient viewpoint belongs to Kania's general programme of coding
separable structures by kernels~\cite{Kan2026Pol}. In that framework,
Polishness of admissible kernel spaces and continuity of quotient-norm
coordinates are established abstractly; Banach lattices and $\FBL[\ell_1]$
are explicitly included among the examples. We give direct proofs of the
lattice-specific closure tests, and use them to compare the two presentations,
analyse concrete lattice operations, and obtain sharp geometric complexity bounds.
The new arguments concern the uniform comparison of the two lattice
presentations, the operations on their codes, and the complexity of
specific lattice properties.

Unless otherwise specified, hyperspaces of a Banach space carry the Wijsman
topology associated with its norm. Write $\Subl(E)$ and $\Ideal(E)$ for the
closed vector sublattices and closed order ideals of a Banach lattice $E$.
Our basic parameter-space result is as follows.

\begin{maintheorem}\label{main:A}
For every separable Banach lattice $E$, both $\Subl(E)$ and $\Ideal(E)$ are
closed subspaces of the Wijsman hyperspace of non-empty closed subsets of $E$.
In particular, they are Polish. The spaces $\Subl(\mathcal C)$ and
$\Ideal(P)$ represent all separable real Banach lattices up to lattice isometry.
\end{maintheorem}

The distinction between an object and its chosen presentation matters when
comparing these spaces. A map that merely assigns an isometric quotient need
not distinguish two different copies of the same lattice. The next result
retains the entire sublattice code.

\begin{maintheorem}\label{main:B}
For every separable Banach lattice $E$ there is a topological embedding
\[
   \Psi_E:\Subl(E)\longrightarrow\Ideal(P)
\]
such that $P/\Psi_E(Y)$ is lattice isometric to $Y$ for every $Y\in\Subl(E)$.
Conversely, there is a Borel map
\[
 \Omega_P:\Ideal(P)\longrightarrow\Subl(\mathcal C)
\]
with $\Omega_P(J)$ lattice isometric to $P/J$ for every $J\in\Ideal(P)$.
\end{maintheorem}

The proof uses continuous dense selections and additional generators whose
quotient norms recover the Wijsman coordinates of $Y$. Surjectivity follows
from an explicit summable approximation argument. In particular, density of
the chosen generators is used to obtain a metric quotient, not merely a
homomorphism with dense range. The reverse translation uses a fixed
pointwise representation from~\cite{LLOT2019Uni}, together with Borel
families of positive annihilator compacta. Thus the two presentations
have the same invariant Borel sets; these translations alone do not
ensure preservation of Borel ranks in both directions.

For operations on codes, a common mechanism is useful: the distance to the
closure of a continuously varying countable set is upper semicontinuous.
This gives the following bounds. Here $\boldsymbol{\Sigma}^0_2$-measurable
means that inverse images of open sets are $F_\sigma$.

\begin{maintheorem}\label{main:C}
The closed ideal-hull map and the maps generating closed sublattices and
closed ideals from sequences are $\boldsymbol{\Sigma}^0_2$-measurable and
have $G_\delta$ graphs. For separable Banach lattices $E,F$, the map
\[
 (I,J)\longmapsto I\widehat\otimes_{|\pi|}J
\]
from $\Ideal(E)\times\Ideal(F)$ to
$\Ideal(E\widehat\otimes_{|\pi|}F)$ has the same properties. Each of these
maps has a dense $G_\delta$ set of points of continuity.
\end{maintheorem}

The tensor statement concerns the tensor product of the ideals themselves.
We also discuss separately the kernel representing the tensor product of the
quotients; these are different operations on ideal codes.

Finally, we apply the coding spaces to order geometry. Our moduli are the
usual monotonicity and order-smoothness moduli; careful normalisation gives
semicontinuity directly, without a loss of Borel rank from Borel selectors.

\begin{maintheorem}\label{main:D}
In $\Subl(E)$, for every separable Banach lattice $E$, uniform monotonicity
and order uniform smoothness are $\boldsymbol\Pi^0_3$ properties. Both
properties are $\boldsymbol\Pi^0_3$-complete in $\Subl(\mathcal C)$ and
in $\Ideal(P)$. Strict monotonicity, order smoothness, and order continuity
are coanalytic in both universal codings.
\end{maintheorem}

The upper bounds follow from semicontinuity of normalised modulus tests.
For the lower bounds we use $\ell_2$-sums of two-dimensional lattices whose
limiting norm has a flat portion at a prescribed small scale. A fixed
reflexive containing space makes their realisations Wijsman-continuous
inside the universal lattice. Kurc's duality theorem transfers the
construction to order uniform smoothness.

\begin{maintheorem}\label{main:E}
Isometric lattice isomorphism on each of $\Subl(\mathcal C)$ and
$\Ideal(P)$ is Borel bireducible with a complete orbit equivalence
relation of Polish-group actions. In particular, both relations are
analytic and non-Borel, and every isomorphism relation of countable
structures Borel reduces to each of them.
\end{maintheorem}

The lower bound comes from Zielinski's theorem~\cite{Zielinski2016} and
explicit parameterised $C(K)$ constructions. The upper bound uses the
coding theorem for Polish relational structures of Elliott, Farah,
Paulsen, Rosendal, Toms, and T\"ornquist~\cite{EFP+2013}. Addition,
rational scalar multiplication, and modulus are represented by closed
graphs, so the theorem applies directly to our lattice codes. The earlier
countable-structure lower bound of Camerlo--Gao~\cite{CG2001} is thereby
subsumed by the orbit classification.

The recognition problem is much simpler for several classical classes.
Here the isometries preserve the lattice structure, including when $p=2$.

\begin{maintheorem}\label{main:F}
Let $E$ be a separable Banach lattice and $1\leqslant p<\infty$.
In $\Subl(E)$, the classes of abstract $L_p$-spaces and AM-spaces are
closed. The lattice-isometry class of $L_p[0,1]$ is $G_\delta$,
and those of $\ell_p$ and $c_0$ are $\boldsymbol\Pi^0_3$.
Within the abstract $L_p$ class, having at least $n$ atoms is $F_\sigma$
for every finite $n\geqslant1$.
\end{maintheorem}

These bounds complement the results on linear isometry classes
in~\cite{CDDK2024Complexity}. Their proofs use disjointisation, a
continuous approximation test for balanced splitting, and total
boundedness of order intervals. We also explain why every non-empty
compact metrisable space belongs to the SICK class introduced
in~\cite{AMRT2022Sick}, and leave the Fatou, Nakano, and lifting
properties for further work; recent relevant results include
\cite{ABRT2025Nakano,KN2026Semi}.

Sections~\ref{sec:prelim-dst}--\ref{sec:prelim-lattice} collect the
background. Section~\ref{sec:coding} treats the parameter spaces and
generation operations, Section~\ref{sec:tensor} treats tensor products,
Section~\ref{sec:equiv} proves Theorem~\ref{main:B},
Section~\ref{sec:classification} proves Theorem~\ref{main:E}, and
Section~\ref{sec:further} proves Theorems~\ref{main:D} and~\ref{main:F}
and records the remaining questions.

\section{Descriptive and hyperspace preliminaries}\label{sec:prelim-dst}

A Polish space is a separable completely metrisable space. We use the usual
Borel classes $\boldsymbol{\Sigma}^0_\alpha$ and
$\boldsymbol{\Pi}^0_\alpha$: in particular,
$\boldsymbol{\Sigma}^0_2=F_\sigma$, $\boldsymbol{\Pi}^0_2=G_\delta$, and
$\boldsymbol{\Pi}^0_3=F_{\sigma\delta}$. Analytic sets are projections of
Borel sets in products of Polish spaces, and coanalytic sets are their
complements. Completeness, when discussed, is with respect to continuous
reductions from zero-dimensional Polish spaces. Standard background can be
found in~\cite{Kechris1995,Dod2010Ban}. An equivalence relation $R$ on a
standard Borel space $X$ is Borel reducible to an equivalence relation $S$
on $Y$ if there is a Borel map $f:X\to Y$ such that
$x\mathrel R x'$ if and only if $f(x)\mathrel S f(x')$.

Let $\mathcal F(X)$ denote the non-empty closed subsets of a metric space
$(X,d)$. For an open set $U\subseteq X$, put
\[
 E^+(U)=\{F\in\mathcal F(X):F\cap U\ne\varnothing\}.
\]
The Wijsman topology is the weakest topology making
$F\mapsto d(x,F)$ continuous for each $x\in X$. If $(a_n)$ is dense in $X$,
it is already generated by the countably many coordinates $d(a_n,F)$.
It is independent of the choice of dense sequence, but can depend on the
compatible metric. When $X$ is Polish, its Wijsman hyperspace is Polish;
see Costantini~\cite{Cost1995}. In this paper, the metric on a Banach space
will always be its norm metric. If $X$ is compact metric, the Wijsman,
Vietoris and Hausdorff-metric topologies on $\mathcal F(X)$ coincide.
We use this fact for $X=\Delta$ and for the Hilbert cube $\mathbb H$.

Following~\cite{God2018Cod}, call a Polish topology $\tau$ on
$\mathcal F(X)$ \emph{admissible} if:
\begin{enumerate}[label=(A\arabic*)]
\item\label{A1} every $E^+(U)$ is $\tau$-open;
\item\label{A2} $\tau$ has a subbase whose members are countable unions of
sets $E^+(U)\setminus E^+(V)$, with $U,V$ open in $X$.
\end{enumerate}
Their Borel sets form the Effros--Borel $\sigma$-algebra. Every open set for
one admissible topology is $F_\sigma$ for another: apply (A2), express open
sets in a metric space as $F_\sigma$, and use second countability.
When needed we impose the extra condition
\begin{enumerate}[start=3,label=(A\arabic*)]
\item\label{A3} $\{(x,F):x\in F\}$ is closed in $X\times\mathcal F(X)$.
\end{enumerate}
Wijsman topology satisfies all three conditions. Indeed,
\[
 |d(x,F)-d(y,F)|\leqslant d(x,y)
\]
makes $(x,F)\mapsto d(x,F)$ jointly continuous, and membership is its zero
set. For (A2), the sets $\{F:d(x,F)<r\}$ are hit sets, and
$\{F:d(x,F)>r\}$ are countable unions of complements of hit sets of balls.

\begin{lemma}\label{lem:wijsman-limit}
Suppose a Polish hyperspace topology satisfies \ref{A1} and \ref{A3}, and
$F_n\to F$. Every $x\in F$ is the limit of some sequence $x_n\in F_n$.
Conversely, $x_n\in F_n$ and $x_n\to x$ imply $x\in F$.
\end{lemma}
\begin{proof}
For each $k$, the open hit condition gives
$F_n\cap B(x,1/k)\ne\varnothing$ for all sufficiently large $n$.
A diagonal choice gives the first assertion. The second is \ref{A3}.
\end{proof}

\begin{lemma}\label{lem:closure-map-borel}
For a Polish space $X$, the map
\[
 X^{\N}\longrightarrow\mathcal F(X),\qquad
 (x_n)\longmapsto\overline{\{x_n:n\in\N\}}
\]
is Borel for the Effros--Borel structure on the target.
\end{lemma}
\begin{proof}
The inverse image of $E^+(U)$ is $\bigcup_n\{(x_j):x_n\in U\}$, which is
open. These hit sets generate the Effros--Borel structure.
\end{proof}

\section{Banach-lattice preliminaries}\label{sec:prelim-lattice}

All spaces and lattices are real. We write $\omega=\{0,1,2,\ldots\}$ and
$\N=\{1,2,\ldots\}$. A Banach lattice is a Banach space with
a vector lattice order for which $|x|\leqslant|y|$ implies
$\|x\|\leqslant\|y\|$. A lattice homomorphism is a linear map preserving
modulus; between Banach lattices it is automatically bounded. We use
\cite{Scha1974Ban} for the standard theory. The elementary estimates
\begin{equation}\label{eq:lattice-Lipschitz}
 \big\||x|-|y|\big\|\leqslant\|x-y\|,
 \qquad
 \|(x\wedge y)-(x'\wedge y')\|
 \leqslant\|x-x'\|+\|y-y'\|
\end{equation}
follow from vector-lattice inequalities; the same estimate holds for
$\vee$. Thus all the lattice operations we use are continuous.

A vector sublattice is a linear subspace stable under modulus. An order ideal
is a solid linear subspace: if $y\in I$ and $|x|\leqslant|y|$, then $x\in I$.
Ideals need not be closed. We write $\Subl(E)$ and $\Ideal(E)$ only for
\emph{closed} sublattices and ideals. The quotient by a closed ideal is a
Banach lattice, the quotient map is a lattice homomorphism, and
\[
 \|x+I\|_{E/I}=\dist(x,I).
\]
For $u\in E_+\setminus\{0\}$, the principal ideal is
$E_u=\{x:|x|\leqslant c u\text{ for some }c>0\}$; it need not be closed
in the original norm.

For a Banach space $X$, its free Banach lattice $\FBL[X]$ comes with a
linear isometry $\delta_X:X\to\FBL[X]$ such that every bounded linear
$T:X\to E$ extends uniquely to a lattice homomorphism
$\widehat T:\FBL[X]\to E$ with
$\widehat T\delta_X=T$ and $\|\widehat T\|=\|T\|$.
See~\cite{aviles2018free,OTTT2026Free,OTTT2026Corr}. The lattice generated by
$\delta_X(X)$ is dense. In particular, for
\[
 P=\FBL[\ell_1],\qquad g_n=\delta_{\ell_1}(e_n)\quad(n\in\omega),
\]
the rational lattice polynomials in the $g_n$ form a countable dense subset
of $P$. Here such polynomials use addition, rational scalar multiplication,
and modulus, with no non-zero constant term. Thus $P$ is separable.

\begin{lemma}\label{lem:metric-quotient}
Let $(u_n)$ lie in the closed unit ball of a Banach space $Y$, and suppose
that its closure contains the unit sphere of $Y$. The operator
$T:\ell_1\to Y$ defined by $Te_n=u_n$ is a metric quotient:
for every $y\in Y$ and $\eta>0$ there is $a\in\ell_1$ such that
$Ta=y$ and $\|a\|_1<\|y\|+\eta$.
If $Y$ is a Banach lattice, then $\widehat T:P\to Y$ is also a metric
quotient and $P/\ker\widehat T$ is lattice isometric to $Y$.
\end{lemma}
\begin{proof}
For $y\ne0$, choose $0<\theta<1$ such that
$\|y\|/(1-\theta)<\|y\|+\eta$. Starting with $r_0=y$, whenever
$r_k\ne0$ choose $u_{n_k}$ with
$\|r_k/\|r_k\|-u_{n_k}\|<\theta$ and put
$r_{k+1}=r_k-\|r_k\|u_{n_k}$. Then
$\|r_k\|\leqslant\theta^k\|y\|$ and
$a=\sum_k\|r_k\|e_{n_k}$ converges in $\ell_1$, with $Ta=y$ and the
required bound. A zero remainder terminates the construction. The case
$y=0$ is immediate. For the extension, $\widehat T\delta_{\ell_1}(a)=y$
and $\|\delta_{\ell_1}(a)\|=\|a\|_1$. Contractivity gives the reverse
quotient-norm inequality. The induced bijective lattice homomorphism is
therefore an isometry.
\end{proof}

Taking a dense sequence in the unit ball shows that every separable Banach
lattice is a lattice-isometric quotient of $P$. This is the quotient
generator in~\cite[Remark~5.1(c)]{Kan2026Pol}. On the embedding side,
Leung, Li, Oikhberg, and Tursi~\cite{LLOT2019Uni} prove that
$\mathcal C=C(\Delta;L_1(0,1))$ contains a lattice-isometric copy of every
separable Banach lattice. Both coding spaces include the zero lattice.

\section{Coding spaces and generation operations}\label{sec:coding}

Write $\SB(E)$ for the closed linear subspaces of $E$.
The following direct closure argument also explains why the closedness of
the membership relation matters. Its Wijsman subspace part is already
contained in~\cite[Proposition~2.4]{Kan2026Pol}.

\begin{theorem}\label{thm:closed-codes}
For a separable Banach lattice $E$, the sets $\SB(E)$, $\Subl(E)$, and
$\Ideal(E)$ are closed in every Polish hyperspace topology satisfying
\ref{A1} and \ref{A3}. In particular, they are Wijsman-closed and Polish.
\end{theorem}
\begin{proof}
Suppose $F_n\in\SB(E)$ and $F_n\to F$. Since $0\in F_n$, we have
$0\in F$. Given $x,y\in F$, Lemma~\ref{lem:wijsman-limit} provides
$x_n,y_n\in F_n$ converging to $x,y$. For scalars $a,b$, the vectors
$ax_n+by_n\in F_n$ converge to $ax+by$, which belongs to $F$ by the same
lemma. Thus $F$ is a linear subspace. If each $F_n$ is a sublattice, then
$|x_n|\to|x|$, so $|x|\in F$.

If each $F_n$ is an ideal, we already know that $F$ is a sublattice.
For $0\leqslant u\leqslant v\in F$, choose $v_n\in F_n$ with $v_n\to v$.
Then $u\wedge|v_n|\in F_n$ and
$u\wedge|v_n|\to u\wedge v=u$, so $u\in F$. If $|x|\leqslant|y|$
with $y\in F$, apply this observation to $x^+$ and $x^-$, with
$v=|y|$, to obtain $x\in F$. Hence $F$ is solid.
\end{proof}

\begin{corollary}\label{cor:ambient-complexity}
For every admissible topology on $\mathcal F(E)$, the sets $\SB(E)$,
$\Subl(E)$ and $\Ideal(E)$ are $G_\delta$, hence Polish. The assertion
for $\SB(E)$ holds for every separable Banach space $E$.
\end{corollary}
\begin{proof}
They are Wijsman-closed, and every Wijsman-open set is $F_\sigma$ for an
arbitrary admissible topology by \ref{A2}.
\end{proof}

The preceding theorem and the two universal representations prove
Theorem~\ref{main:A}. The next result is the continuous-selection principle
used in~\cite{God2018Cod}. We include a proof, so that its validity for an
arbitrary admissible topology does not depend on continuity of the
norm-distance coordinates in that topology.

\begin{theorem}\label{thm:selection}
If $E$ is a separable Banach space, then $\SB(E)$, with any admissible
topology, admits continuous maps $f_j:\SB(E)\to E$ such that
$f_j(F)\in F$ and $\{f_j(F):j\in\omega\}$ is dense in $F$ for every $F$.
\end{theorem}
\begin{proof}
Write $Z=\SB(E)$ and fix a compatible metric $d_Z$ on $Z$. For each
$a$ in a fixed countable dense subset of $E$ and each rational $r>0$, put
\[
 U_{a,r}=\{F:F\cap B(a,r)\ne\varnothing\},\qquad
 H_{a,r}(F)=F\cap\overline B(a,r)\quad(F\in U_{a,r}).
\]
The set $U_{a,r}$ is open by~\ref{A1}, and $H_{a,r}$ has non-empty closed
convex values. It is lower semicontinuous. Indeed, if an open set $V$
meets $H_{a,r}(F)$ at a boundary point of the ball, perturb that point
slightly towards a point of $F\cap B(a,r)$. This gives
$w\in V\cap F\cap B(a,r)$. A sufficiently small ball about $w$ is contained
in $V\cap B(a,r)$, and~\ref{A1} ensures that nearby subspaces meet it.
Michael's theorem~\cite{Michael1956} therefore provides a continuous
selection $s_{a,r}:U_{a,r}\to E$ with
$\|s_{a,r}(F)\|\leqslant\|a\|+r$.

For $l\in\N$, multiply this selection by
\[
 \chi_l(F)=\min\{1,l\,d_Z(F,Z\setminus U_{a,r})\}
\]
and extend the result by zero outside $U_{a,r}$. If $U_{a,r}=Z$, use
$\chi_l=1$; empty $U_{a,r}$ may be omitted. The bound on $s_{a,r}$ gives
continuity of the zero extension at the boundary, and all its values lie
in the corresponding subspace. These countably many global selections
are dense in each $F$: given $x\in F$ and $\varepsilon>0$, choose
$\|a-x\|<\varepsilon/4$ and rational
$\|a-x\|<r<\varepsilon/2$. Then $F\in U_{a,r}$ and, for sufficiently
large $l$, $\chi_l(F)=1$, while
$\|s_{a,r}(F)-x\|\leqslant r+\|a-x\|<\varepsilon$.
Include the zero selection and enumerate the resulting family.
\end{proof}

\begin{proposition}\label{prop:dense-tests}
Fix such selections and a dense sequence $(a_k)$ in a separable Banach
lattice $E$. A closed subspace $F$ is a sublattice precisely when
$|f_j(F)|\in F$ for all $j$. It is an ideal precisely when, in addition,
$|a_k|\wedge|f_j(F)|\in F$ for every $j,k$.
\end{proposition}
\begin{proof}
The necessity is immediate. For sufficiency, approximate an arbitrary
$x\in F$ by selections and use continuity of modulus. If
$0\leqslant u\leqslant v\in F$, choose $a_{k_n}\to u$ and
$f_{j_n}(F)\to v$. The tested elements
$|a_{k_n}|\wedge|f_{j_n}(F)|$ converge to $u$. Closedness and decomposition
into positive and negative parts give solidity.
\end{proof}

\subsection{Continuously generated closed sets}\label{subsec:continuous-generation}

The following observation gives a common complexity bound for the operations
used below. All Wijsman topologies on Banach spaces in this paper are induced
by the norm metric.

\begin{lemma}\label{lem:continuous-generation}
Let $\mathscr S$ be a Polish space, let $(M,d)$ be a separable metric space, and let
$s_m\colon \mathscr S\to M$ be continuous for $m\in\omega$. Suppose that
\[
 H(p)=\overline{\{s_m(p):m\in\omega\}}
\]
belongs to a subspace $\mathscr H$ of the Wijsman hyperspace of $M$ for every
$p\in \mathscr S$. Then $H\colon \mathscr S\to\mathscr H$ is
$\boldsymbol\Sigma^0_2$-measurable and its graph is $G_\delta$ in
$\mathscr S\times\mathscr H$. Moreover, $H$ is continuous at every point of a dense
$G_\delta$ subset of $\mathscr S$.
\end{lemma}

\begin{proof}
Fix a dense sequence $(a_n)$ in $M$. For each $n$,
\[
 h_n(p):=d(a_n,H(p))=\inf_m d(a_n,s_m(p))
\]
is upper semicontinuous. Thus $\{h_n<r\}$ is open and
\[
 \{h_n>r\}=\bigcup_{\substack{q\in\bbQ\\q>r}}\{h_n\geqslant q\}
\]
is $F_\sigma$. The Wijsman topology has a countable base formed from finite
intersections of such coordinate conditions. Consequently the inverse image
of each open set is $F_\sigma$.

On $\mathscr S\times\mathscr H$, put $a(p,K)=d(a_n,K)$ and $b(p,K)=h_n(p)$.
The function $a$ is continuous and $b$ is upper semicontinuous. Hence
$\{a\leqslant b\}$ is closed, whereas
\[
 \{b\leqslant a\}=\bigcap_{r\geqslant1}\{b<a+1/r\}
\]
is $G_\delta$. It follows that $\{a=b\}$ is $G_\delta$. Equality of all
countably many distance coordinates characterises equality of closed sets,
so $\operatorname{Graph}(H)$ is $G_\delta$.

Finally, each real-valued upper semicontinuous function $h_n$ is continuous
at every point of a dense $G_\delta$ set. Intersecting these sets over $n$
and using the initial definition of the Wijsman topology proves the last
assertion. This argument does not require a general identification of
$\boldsymbol\Sigma^0_2$-measurability with Baire class~$1$ for maps into
arbitrary Polish spaces.
\end{proof}

\subsection{The ideal hull}\label{subsec:ideal-hull}

Fix continuous dense selections $f_j\colon\SB(E)\to E$ and a dense sequence
$(g_k)$ in a separable Banach lattice $E$.

\begin{proposition}\label{prop:ideal-hull}
For $F\in\Subl(E)$, the closed ideal generated by $F$ is
\[
 \Gamma(F)=\overline{\operatorname{span}}
 \{\,|g_k|\wedge|f_j(F)|:j,k\in\omega\,\}.
\]
The map $\Gamma\colon\Subl(E)\to\Ideal(E)$ is
$\boldsymbol\Sigma^0_2$-measurable, has a $G_\delta$ graph, and is continuous
at every point of a dense $G_\delta$ subset of $\Subl(E)$.
\end{proposition}

\begin{proof}
The algebraic ideal generated by the sublattice $F$ is
\[
 I(F)=\{x\in E: |x|\leqslant v\text{ for some }v\in F^+\}.
\]
Every displayed generator of $\Gamma(F)$ belongs to $I(F)$, so $\Gamma(F)$
is contained in its closure. Conversely, if $0\leqslant u\leqslant v\in F^+$,
choose $g_{k_n}\to u$ and $f_{j_n}(F)\to v$. Continuity of modulus and meet
gives
\[
 |g_{k_n}|\wedge|f_{j_n}(F)|\longrightarrow u\wedge v=u.
\]
Decomposing an element of $I(F)$ into its positive and negative parts shows
that $I(F)\subseteq\Gamma(F)$. This proves the formula. Enumerating finite
rational linear combinations of the displayed generators gives continuous
maps with dense values in $\Gamma(F)$. Lemma~\ref{lem:continuous-generation}
proves all the topological assertions.
\end{proof}

\begin{example}\label{ex:ideal-hull-discontinuous}
The ideal-hull map need not be continuous. In $\R^2$ with coordinatewise
order, let $F_t=\operatorname{span}\{(1,t)\}$ for $t\geqslant0$. These are
sublattices and $F_t\to F_0$ in the norm-Wijsman topology as $t\downarrow0$.
However, $\Gamma(F_t)=\R^2$ for $t>0$, whereas $\Gamma(F_0)=F_0$.
\end{example}

\subsection{Generation of sublattices and ideals}\label{subsec:borel-gen}

\begin{proposition}\label{prop:borel-generated-sublattice}
Let $E$ be a separable Banach lattice. The maps from $E^\omega$ which assign
to a sequence its generated closed sublattice and its generated closed ideal
are $\boldsymbol\Sigma^0_2$-measurable, have $G_\delta$ graphs, and are
continuous at every point of respective dense $G_\delta$ subsets.
\end{proposition}

\begin{proof}
Enumerate all lattice terms with rational coefficients, formed from finitely
many coordinate variables using addition, rational scalar multiplication,
and modulus. Each term is a continuous function on $E^\omega$. The closure
of their values is precisely the generated closed sublattice: it contains
the generators and is closed under the lattice operations and real scalar
multiplication, the latter by approximation with rational scalars.
Lemma~\ref{lem:continuous-generation} applies.

For ideal generation, fix a dense sequence $(g_k)$ in $E$ and put
\[
 v_m(x)=\sum_{j=0}^{m}|x_j|,\qquad
 h_{m,k,r}(x)=|g_k|\wedge r v_m(x)
 \quad(m,k\in\omega,\ r\in\N).
\]
All these functions are continuous. Their closed linear span is the closed
ideal generated by the $x_j$. Indeed, the algebraic ideal consists of the
vectors $z$ satisfying $|z|\leqslant r v_m(x)$ for some $m,r$; positive such
vectors are limits of $h_{m,k,r}(x)$ as $g_k$ tends to that vector. Taking
positive and negative parts proves the assertion. Enumerate the rational
linear combinations of the $h_{m,k,r}$ and apply
Lemma~\ref{lem:continuous-generation} again. In particular, this proof gives
the stated complexity directly; composition of two
$\boldsymbol\Sigma^0_2$-measurable maps alone would not suffice.
\end{proof}

\section{Borel complexity of the Fremlin tensor product}\label{sec:tensor}

\subsection{The tensor product and closed ideals}

For Banach lattices $E,F$, write $E\hat\otimes_{|\pi|}F$ for the Fremlin
projective tensor product. The algebraic tensor product embeds in the
Fremlin vector lattice tensor product, and its norm is
\[
 \|u\|_{|\pi|}=\inf\left\{\sum_{j=1}^n\|x_j\|\|y_j\|:
 x_j\in E^+,\ y_j\in F^+,\ |u|\leqslant\sum_{j=1}^nx_j\otimes y_j\right\}.
\]
Its completion is a Banach lattice. We use the cross-norm identity
$\|x\otimes y\|_{|\pi|}=\|x\|\|y\|$ and the universal property:
a bounded positive bilinear map into a Banach lattice extends uniquely to a
positive operator of the same norm on $E\hat\otimes_{|\pi|}F$; a lattice
bimorphism extends to a lattice homomorphism. See Fremlin~\cite[Theorem~1E(iii) and~1F]{Frem1974Ban} and
Puglisi~\cite{Pug2007Fremlin}. The positive cone is the closure of the cone
of finite sums of positive elementary tensors. Indeed, otherwise a
continuous linear functional separating a positive vector from that closed
cone would induce a positive bilinear form on $E\times F$ while failing
to be positive on the tensor product, contrary to the universal property.

We include the ideal property in the separable case to make clear why the
operations below take values in spaces of closed order ideals.

\begin{lemma}\label{lem:tensor-ideal-property}
Let $E,F$ be separable Banach lattices and let $I\in\Ideal(E)$,
$J\in\Ideal(F)$. The canonical map
\[
 I\hat\otimes_{|\pi|}J\longrightarrow E\hat\otimes_{|\pi|}F
\]
is an isometric lattice embedding whose range is a closed order ideal.
\end{lemma}

\begin{proof}
First we construct positive contractions $S_n\colon E\to I$ such that
$0\leqslant S_nx\leqslant x$ for $x\in E^+$ and $S_nx\to x$ for $x\in I$.
If $I=\{0\}$ or $J=\{0\}$, the assertion is immediate; hence assume both are non-zero. Choose a positive element $v$ whose
principal ideal is dense in $I$, and choose a positive element $u\geqslant v$
whose principal ideal $E_u$ is dense in $E$. Such elements exist by
separability: take a norm-convergent strictly positive weighted sum of a
dense sequence of positive vectors. With its order-unit norm, $E_u$ is
lattice isometric to $C(K)$ for some compact $K$, with $u$ corresponding to
$\mathbf1$; see~\cite{Scha1974Ban}. On $E_u$, define $S_n$ as multiplication
by $\mathbf1\wedge nv$. Then $|S_nx|\leqslant|x|$, so $S_n$ extends to a
positive contraction on $E$. Moreover,
$|S_nx|\leqslant n\|x\|_u v$ for $x\in E_u$, whence $S_n(E)\subseteq I$.
The scalar inequality $0\leqslant t(1-\min\{1,nt\})\leqslant1/(4n)$ gives
\[
 0\leqslant v-S_nv\leqslant u/(4n).
\]
It follows first that $S_nx\to x$ on the principal ideal generated by $v$,
and then, by boundedness, on all of $I$. Construct $T_n\colon F\to J$
similarly.

Let $i$ denote the canonical contractive lattice homomorphism from the
smaller tensor product into the larger one. The universal property gives
contractions
\[
 R_n=S_n\otimes T_n:E\hat\otimes_{|\pi|}F
       \longrightarrow I\hat\otimes_{|\pi|}J.
\]
For every $w$ in the smaller tensor product, $R_niw\to w$: this holds
first on elementary tensors in $I\otimes J$ and then by density and uniform boundedness. Thus
$\|w\|=\lim_n\|R_niw\|\leqslant\|iw\|$, proving that $i$ is isometric.
Its range $L$ is therefore a closed sublattice.

Put $P_n=iR_n$. For $x,y\geqslant0$,
\[
 x\otimes y-S_nx\otimes T_ny
 =(x-S_nx)\otimes y+S_nx\otimes(y-T_ny)\geqslant0.
\]
The positive bilinear extension property therefore implies
$0\leqslant P_n\leqslant\mathrm{id}$ on the entire tensor product. If $0\leqslant z\leqslant w\in L$, then
\[
 0\leqslant z-P_nz\leqslant w-P_nw,
 \qquad \|w-P_nw\|\longrightarrow0.
\]
Since $P_nz\in L$ and $L$ is closed, $z\in L$. This proves solidity.
\end{proof}

\subsection{Tensoring ideals and tensoring their quotients}

The descriptive argument follows the fixed-coordinate method of
Kania~\cite[Proposition~9.2]{Kan2026Pol}: continuously varying generators
make each distance-to-kernel coordinate upper semicontinuous, and the
Wijsman topology converts these coordinate estimates into Borel bounds.
Here the lattice ideal property and the quotient identity below provide
the additional structural input.

The following two operations have different meanings. Tensoring two ideals
as Banach lattices gives $I\hat\otimes_{|\pi|}J$. If the ideals instead
encode the quotients $E/I$ and $F/J$, the appropriate output is the kernel
of the tensor product of the quotient maps.

\begin{theorem}\label{thm:tensor}
Let $E,F$ be separable Banach lattices and put
$Z=E\hat\otimes_{|\pi|}F$. Both maps
\[
\begin{aligned}
 \Xi:\Ideal(E)\times\Ideal(F)&\longrightarrow\Ideal(Z),
 &\Xi(I,J)&=I\hat\otimes_{|\pi|}J,\\
 \Theta:\Ideal(E)\times\Ideal(F)&\longrightarrow\Ideal(Z),
 &\Theta(I,J)&=\overline{I\otimes F+E\otimes J}^{\,Z}
\end{aligned}
\]
are $\boldsymbol\Sigma^0_2$-measurable and have $G_\delta$ graphs for the
norm-Wijsman topologies. Each is continuous at every point of a dense
$G_\delta$ subset of its domain. Moreover, there is a canonical isometric
lattice isomorphism
\begin{equation}\label{eq:tensor-quotient}
 Z/\Theta(I,J)\cong(E/I)\hat\otimes_{|\pi|}(F/J).
\end{equation}
\end{theorem}

\begin{proof}
Let $f_i(I)$ and $g_j(J)$ be continuous dense selections, and fix dense
sequences $(a_k)$ in $E$ and $(b_l)$ in $F$. Finite rational linear
combinations of
\[
 f_i(I)\otimes g_j(J)
\]
are continuous functions of $(I,J)$, with values dense in $\Xi(I,J)$.
Indeed, the cross-norm estimate makes $(x,y)\mapsto x\otimes y$ jointly
continuous, so dense sets in the factors generate a dense linear span in
their completed tensor product. By
Lemma~\ref{lem:tensor-ideal-property}, these tensor products are closed
ideals of $Z$. Lemma~\ref{lem:continuous-generation} applies.

The same argument applies to rational linear combinations of
\[
 f_i(I)\otimes b_l\quad\text{and}\quad a_k\otimes g_j(J).
\]
Their closure is $\Theta(I,J)$. By
Lemma~\ref{lem:tensor-ideal-property}, the closures of $I\otimes F$ and
$E\otimes J$ are ideals; the closure of their sum is an ideal as well.
This proves all the complexity and continuity assertions.

It remains to verify~\eqref{eq:tensor-quotient}. Write $q_I,q_J$ for the
quotient maps and $q:Z\to Z/\Theta(I,J)$. The bilinear map
\[
 (x,y)\longmapsto q_Ix\otimes q_Jy
\]
is a contractive lattice bimorphism. Its extension to $Z$ vanishes on
$\Theta(I,J)$, and hence induces a contractive lattice homomorphism
\[
 A:Z/\Theta(I,J)\longrightarrow(E/I)\hat\otimes_{|\pi|}(F/J).
\]
In the other direction, the formula
\[
 b(q_Ix,q_Jy)=q(x\otimes y)
\]
is well defined and bilinear: changing a lift adds an element of
$I\otimes F+E\otimes J$. Taking infima over lifts gives
\[
 \|b(\xi,\eta)\|\leqslant\|\xi\|\|\eta\|.
\]
Every positive quotient vector has a positive lift, so $b$ is positive.
It is a lattice bimorphism because quotient maps preserve modulus and
$|x\otimes y|=|x|\otimes|y|$. Thus it extends to a contractive lattice
homomorphism $B$ in the reverse direction. The maps $AB$ and $BA$ are the
identity on dense sets of elementary tensors and their quotient images.
They are therefore the identity everywhere, proving the isometric lattice
isomorphism.
\end{proof}

\begin{corollary}\label{cor:tensor-sections}
Each section of either $\Xi$ or $\Theta$, with one ideal fixed, is
$\boldsymbol\Sigma^0_2$-measurable, has a $G_\delta$ graph, and is continuous
at every point of a dense $G_\delta$ subset of its domain.
\end{corollary}

\begin{proof}
Apply Lemma~\ref{lem:continuous-generation} to the same continuous generators
with one variable fixed. A comeagre set in a product need not have comeagre
sections at every point, so the section assertion is proved separately.
\end{proof}

Taking $E=F=\FBL[\ell_1]$, the map $\Theta$ supplies a Borel tensor-product
operation on quotient codes, with output coded over
$\FBL[\ell_1]\hat\otimes_{|\pi|}\FBL[\ell_1]$. The use of this different
fixed ambient lattice is explicit. No identification of
$\Xi(I,J)$ with the kernel of a tensor product of quotients is intended.

\section{A continuous bridge between the coding spaces}\label{sec:equiv}

The sublattice and quotient presentations are related more closely than mere
universality suggests.  We prove that the sublattice coding space admits a
topological embedding into the ideal coding space, preserving the represented
Banach lattice isometrically.  The construction works for every separable
ambient Banach lattice.  Throughout this section the hyperspaces carry their
norm-Wijsman topologies.

\begin{lemma}\label{lem:distance-marker}
Let $E$ be a separable Banach space and fix $a\in E$.  There is a continuous
map $v_a\colon\SB(E)\to E$ such that, for every $Y\in\SB(E)$,
\[
 v_a(Y)\in Y,\qquad
 \|v_a(Y)\|=
 \frac{\|a\|-\dist(a,Y)}{1+\|a\|}.
\]
\end{lemma}

\begin{proof}
There is nothing to prove when $a=0$.  Put $d(Y)=\dist(a,Y)$ and
$A=\{Y:d(Y)<\|a\|\}$.  This is an open subspace of $\SB(E)$.
For $Y\in A$, let
\[
 r(Y)=\frac{\|a\|+d(Y)}2,\qquad
 H(Y)=Y\cap\overline B(a,r(Y)).
\]
The values of $H$ are non-empty closed convex sets, because
$r(Y)>d(Y)$.  We verify lower semicontinuity.  Fix $Y_0\in A$ and an open
set $U\subseteq E$ meeting $H(Y_0)$.  Choose $u\in U\cap H(Y_0)$ and
$z\in Y_0$ with $\|a-z\|<r(Y_0)$.  A sufficiently small convex
perturbation of $u$ towards $z$ gives $w\in U\cap Y_0$ satisfying
$\|a-w\|<r(Y_0)$.  Choose $\eta>0$ such that
\[
 B(w,\eta)\subseteq U,\qquad
 \|a-w\|+2\eta<r(Y_0).
\]
For every $Y$ in the neighbourhood
\[
 \{Y:Y\cap B(w,\eta)\ne\varnothing\}
 \cap\{Y:r(Y)>r(Y_0)-\eta\}
\]
of $Y_0$, the set $H(Y)$ meets $U$.  Here the first set is open by
the lower Vietoris part of the Wijsman topology.  Thus $H$ is lower
semicontinuous.  Since $A$ is metrisable, it is paracompact.  Michael's
continuous selection theorem~\cite{Michael1956} gives a continuous
map $s\colon A\to E$ with $s(Y)\in H(Y)$.  Moreover,
\[
 \|s(Y)\|\geqslant\|a\|-r(Y)
 =\frac{\|a\|-d(Y)}2>0.
\]
Define
\[
 v_a(Y)=
 \begin{cases}
 \displaystyle\frac{\|a\|-d(Y)}{1+\|a\|}
                 \frac{s(Y)}{\|s(Y)\|},&Y\in A,\\[6pt]
 0,&Y\notin A.
 \end{cases}
\]
This map is continuous on $A$.  Since $0\in Y$ always implies
$d(Y)\leqslant\|a\|$, its prescribed norm tends to zero at every
boundary point of $A$.  The extension is therefore continuous everywhere.
\end{proof}

\Needspace{8\baselineskip}
\begin{theorem}\label{thm:coding-embedding}
Let $E$ be a separable Banach lattice and put $P=\FBL[\ell_1]$.
There is a topological embedding
\[
 \Psi_E\colon\Subl(E)\longrightarrow\Ideal(P)
\]
such that $P/\Psi_E(Y)$ is isometrically lattice-isomorphic to $Y$ for
every $Y\in\Subl(E)$.  Its image is a $G_\delta$ subset of
$\Ideal(P)$.
\end{theorem}

\begin{proof}
Choose a dense sequence $(a_j)_{j\in\omega}$ in $E$ and continuous dense
selections $f_i\colon\SB(E)\to E$ from
Theorem~\ref{thm:selection}.  Set
\[
 u_i(Y)=\frac{f_i(Y)}{\max\{1,\|f_i(Y)\|\}},\qquad
 v_j(Y)=v_{a_j}(Y),
\]
where the latter maps are given by Lemma~\ref{lem:distance-marker}.
The maps $u_i$ and $v_j$ are continuous, take values in $Y$, and have norm
at most one.  The sequence $(u_i(Y))_i$ is dense in the closed unit ball
of $Y$, since radial retraction onto that ball is continuous and fixes
every point of the ball.

Define a contractive operator $T_Y\colon\ell_1\to E$ by
\[
 T_Ye_{2i}=u_i(Y),\qquad T_Ye_{2j+1}=v_j(Y),
\]
and let $Q_Y\colon P\to E$ be its contractive lattice-homomorphic
extension.  Its range is contained in $Y$, because $Y$ is a closed
sublattice containing the images of all generators.
Lemma~\ref{lem:metric-quotient} shows that $Q_Y$ is a metric quotient
onto $Y$.  Consequently,
\begin{equation}\label{eq:kernel-distance}
 \dist(z,\ker Q_Y)=\|Q_Yz\|\qquad(z\in P).
\end{equation}
In particular, $\Psi_E(Y):=\ker Q_Y$ is a closed ideal and the induced
map $P/\Psi_E(Y)\to Y$ is an isometric lattice isomorphism.

For every fixed $z\in P$, the map $Y\mapsto Q_Yz$ is norm-continuous.
Indeed, this holds first for lattice polynomials in finitely many
generators, by continuity of the maps $u_i,v_j$ and the lattice
operations.  Such polynomials are dense in $P$, and
\[
 \sup_Y\|Q_Yz-Q_Yp\|\leqslant\|z-p\|
\]
allows passage to an arbitrary $z$.  Equation~\eqref{eq:kernel-distance}
then proves Wijsman continuity of $\Psi_E$.

Write $g_n=\delta_{\ell_1}(e_n)$.  The odd generators recover all ambient
Wijsman coordinates:
\begin{equation}\label{eq:recover-ambient-distance}
 \dist(a_j,Y)=\|a_j\|-(1+\|a_j\|)
                  \dist(g_{2j+1},\Psi_E(Y)).
\end{equation}
Since the coordinates on the left separate closed subsets of $E$,
$\Psi_E$ is injective.  Since they generate the Wijsman topology,
the same formula proves continuity of its inverse on its image.
Thus $\Psi_E$ is a topological embedding.  Finally, its image is a
Polish subspace of the Polish space $\Ideal(P)$, so it is $G_\delta$.
\end{proof}

Taking $E=\mathcal C$ gives a continuous, injective translation from
universal sublattice codes to quotient kernels, with a continuous inverse
on its image.  The marker generators retain the chosen ambient copy of
$Y$, rather than only its abstract isometric lattice-isomorphism class.
The zero lattice is included: its code is the whole ideal $P$.

\begin{remark}
The density assumption in Lemma~\ref{lem:metric-quotient} is stronger
than density of the range of an arbitrary contractive lattice
homomorphism: it provides lifts with norms arbitrarily close to the norm
of the target vector.  In particular, the normalised dense-sphere
construction also gives a surjective metric quotient; no separate
closed-range argument is needed.  The marker coordinates supply the
additional injectivity and inverse continuity.
\end{remark}

For every class $\mathscr A$ invariant under lattice isometry, its sublattice
codes form the continuous inverse image, under $\Psi_E$, of its quotient
codes. Thus Borel upper bounds in the quotient coding transfer without any
increase of rank, and continuous hardness results for sublattice codes
transfer to quotient codes.

\subsection{A Borel translation of arbitrary quotient codes}

\begin{lemma}\label{lem:borel-cantor-surjections}
Let $M$ be a non-empty compact metrisable space. There is a Borel map
\[
 \mathcal F(M)\ni K\longmapsto s_K\in C(\Delta,M)
\]
such that $s_K$ is a continuous surjection from $\Delta$ onto $K$ for
every $K\in\mathcal F(M)$. The function space carries its uniform topology.
\end{lemma}

\begin{proof}
Fix a continuous surjection $\pi\colon\Delta\to M$. The map
$K\mapsto L_K=\pi^{-1}(K)$ is Borel into $\mathcal F(\Delta)$.
Indeed, for an open $U\subseteq\Delta$,
\[
 \{K:L_K\cap U\ne\varnothing\}
 =\bigcup_{[t]\subseteq U}
   \{K:K\cap\pi([t])\ne\varnothing\},
\]
where $[t]$ ranges over the countable family of basic clopen cylinders
contained in $U$. Each set on the right is closed in
$\mathcal F(M)$, since $\pi([t])$ is compact. The hit-open sets generate
the Borel structure of $\mathcal F(\Delta)$, proving the claim.

For a non-empty compact $L\subseteq\Delta$, write
$T_L=\{t\in2^{<\omega}:L\cap[t]\ne\varnothing\}$.
Define $r_L\colon\Delta\to L$ to fix every point of $L$. For
$x\notin L$, retain its longest initial segment in $T_L$ and complete
that segment to the lexicographically least branch of $T_L$ extending
it. This is a continuous retraction: it is locally constant off $L$,
and an input sharing the first $n$ coordinates with a point of $L$
has an output sharing those coordinates as well.

Moreover, $L\mapsto r_L$ is continuous into $C(\Delta,\Delta)$ with
the uniform topology. The first $n$ output coordinates are determined
by the first $n$ input coordinates and the finite tree
$T_L\cap2^{\leqslant n}$. Here $T_L$ is pruned: every node extends to an infinite branch.
This ensures that greedy continuation uses only the stated finite tree.
For each fixed finite string $t$, the
condition $L\cap[t]\ne\varnothing$ is clopen in
$\mathcal F(\Delta)$, since its Wijsman and Vietoris topologies coincide. Thus sufficiently nearby compacta have
retractions whose outputs agree on the first $n$ coordinates,
uniformly over all inputs. Finally, set $s_K=\pi\circ r_{L_K}$.
Composition with the fixed uniformly continuous map $\pi$ preserves
continuity in the uniform topology. Hence $K\mapsto s_K$ is Borel,
and $s_K(\Delta)=\pi(L_K)=K$.
\end{proof}

\begin{theorem}\label{thm:reverse-borel-coding}
Let $G$ be a separable real Banach lattice.  There is a Borel map
\[
 \Omega_G:\Ideal(G)\longrightarrow\Subl(\mathcal C)
\]
such that $\Omega_G(J)$ is lattice isometric to $G/J$ for every
$J\in\Ideal(G)$.  Moreover, function $\Omega_G$ is $\boldsymbol\Sigma^0_4$-measurable: for every set
$E^+(U)=\{Y\in\Subl(\mathcal C): Y\cap U\ne\varnothing\}$, the preimage
$\Omega_G^{-1}(E^+(U))$ is $\boldsymbol\Sigma^0_3$
in $\Ideal(G)$.
\end{theorem}

\begin{proof}
The assertion is immediate when $G=\{0\}$, so suppose that $G\ne\{0\}$.
By \cite[Proposition~2.2]{LLOT2019Uni}, we may fix a separable Banach
lattice $F$ containing $G$ as a closed sublattice and a finitely branching
tree $(x_\sigma)$ in $F_+$ whose linear span is dense in $F$.  The
successors of each $x_\sigma$ are pairwise disjoint and sum to
$x_\sigma$.  After rescaling the whole tree, assume that its root $e$
satisfies $\|e\|\leqslant1$.  Let
\[
 M=(B_{F^*})_+
\]
with the weak-star topology, so that $M$ is compact and metrisable.

We use a pointwise feature of the construction in
\cite[Proposition~2.1]{LLOT2019Uni}.  That construction gives a lattice
isometry $T:F\to C(M,L_1[0,1])$ satisfying
\begin{equation}\label{eq:pointwise-positive-dual-representation}
 \|(Tx)(\varphi)\|_1=\varphi(|x|)
 \qquad(x\in F,\ \varphi\in M).
\end{equation}
Indeed, at $\varphi\in M$ assign to the root the interval
$[0,\varphi(e)]$ and subdivide each assigned interval consecutively into
intervals of lengths $\varphi(x_\tau)$, with $\tau$ running over its
successors in a fixed order.  Send $x_\sigma$ to the characteristic
function of its interval.  These functions depend continuously on
$\varphi$ in the $L_1$ norm.  At every fixed level the intervals are
pairwise disjoint, and therefore the displayed identity holds for every
finite linear combination of tree elements.  Taking the supremum over
$\varphi\in M$ proves that this map is isometric.  Its extension to $F$
is a lattice homomorphism, and the identity persists by norm continuity.

For $J\in\Ideal(G)$ define
\[
 K_J=\{\varphi\in M:\varphi|_J=0\}.
\]
These are non-empty compact subsets of $M$.  The graph of $J\mapsto K_J$
is closed.  To see this, suppose that $J_n\to J$ in the Wijsman topology
and $\varphi_n\to\varphi$ weak-star, where $\varphi_n\in K_{J_n}$.
Given $j\in J$, Wijsman convergence provides $j_n\in J_n$ with
$j_n\to j$ in norm.  As $\|\varphi_n\|\leqslant1$,
\[
 |\varphi_n(j)|\leqslant\|j-j_n\|\longrightarrow0,
\]
and hence $\varphi(j)=0$.
For every closed $C\subseteq M$, compactness of $M$ now shows that
\[
 \{J \in \Ideal(G):K_J\cap C\ne\varnothing\}
\]
is closed. 
 If $V\subseteq M$ is open, then of course $M\setminus V$ is closed, hence the set $\{J: K_J\subseteq V\}=\{J: K_J\cap(M\setminus V)=\varnothing\}$ is
open.  If $U\subseteq M$ is open, write $U=\bigcup_{m=1}^\infty C_m$ with
$C_m$ closed, as $M$ is metrisable. Then
\[
 \{J\in\Ideal(G): K_J\cap U\ne\varnothing\}
 =\bigcup_{m=1}^\infty \{J\in\Ideal(G): K_J\cap C_m\ne\varnothing\}
\]
is an $F_\sigma$ set.  Every open set in the Vietoris topology of
$\mathcal F(M)$ is a countable union of finite intersections of such sets. Consequently, the map $J\mapsto K_J$ is
$\boldsymbol\Sigma^0_2$-measurable.

Next, consider the selection map $K\mapsto s_K$ from
Lemma~\ref{lem:borel-cantor-surjections}.  For an open cylinder $[t]$ in
$\Delta$, the set $\pi([t])$ is compact in $M$.  Thus
$K\mapsto L_K=\pi^{-1}(K)$ pulls subbasis sets back to $F_\sigma$ sets in
$\mathcal F(M)$.  Since $L\mapsto r_L$ is continuous from $\mathcal F(\Delta)$
into $C(\Delta,\Delta)$ and $\pi$ is continuous, $K\mapsto s_K$ is $\boldsymbol\Sigma^0_2$-measurable.

The composition of two $\boldsymbol\Sigma^0_2$-measurable functions acting between Polish
spaces is, in general, $\boldsymbol\Sigma^0_3$-measurable.  Therefore,
the map $J\mapsto s_{K_J}$ from $\Ideal(G)$ into $C(\Delta,M)$ 
is $\boldsymbol\Sigma^0_3$-measurable.

We claim that, for every $x\in G$,
\begin{equation}\label{eq:quotient-norm-positive-extensions}
 \dist(x,J)=\sup_{\varphi\in K_J}\varphi(|x|).
\end{equation}
The inequality from right to left follows because the restriction of
each $\varphi\in K_J$ induces a positive functional of norm at most one
on the quotient Banach lattice $G/J$.  Conversely, take a positive
functional $\psi\in B_{(G/J)^*}$ satisfying
$\psi(|x+J|)=\|x+J\|$ and let $f=\psi\circ q_J\in G^*_+$, where
$q_J:G\to G/J$ is the quotient map.  Then $\|f\|\leqslant1$ and
$f|_J=0$.  The function
\[
 p(z)=\|z^+\|\qquad(z\in F)
\]
is sublinear, and $f(y)\leqslant p(y)$ for every $y\in G$.
The Hahn--Banach theorem therefore gives a linear extension
$\varphi:F\to\mathbb R$ with $\varphi(z)\leqslant p(z)$.
For $z\geqslant0$ this applied to $-z$ yields $\varphi(z)\geqslant0$;
applying the inequality to $z$ and $-z$ also gives
$|\varphi(z)|\leqslant\|z\|$.  Thus $\varphi\in K_J$, and
\[
 \varphi(|x|)=\psi(|x+J|)=\|x+J\|,
\]
which proves the claim.

Use Lemma~\ref{lem:borel-cantor-surjections} to choose, Borelly in
$J$, a continuous surjection $\pi_J:\Delta\to K_J$, regarded as an
element of $C(\Delta,M)$.  Define
\[
 A_Jx=(Tx)\circ\pi_J\in \mathcal C\qquad(x\in G).
\]
This is a lattice homomorphism and, by
\eqref{eq:pointwise-positive-dual-representation} and
\eqref{eq:quotient-norm-positive-extensions},
\[
 \|A_Jx\|_{\mathcal C}=\dist(x,J).
\]
It therefore induces a lattice isometry from $G/J$ onto a closed
sublattice of $\mathcal C$.

 Fix a countable dense sequence $(p_n)_{n\in\N}$ in $G$.  For each
fixed $n$, the composition map $s\mapsto (Tp_n)\circ s$ from
$C(\Delta,M)$ to $\mathcal C$ is continuous by uniform continuity of $Tp_n$.
Hence each generator map
\[
 J\longmapsto A_J p_n
\]
is a composition of a continuous map with a $\boldsymbol\Sigma^0_3$-measurable function, hence is itself $\boldsymbol\Sigma^0_3$-measurable.

Finally, set $\Omega_G(J)=\overline{\{A_J p_n:n\in\N\}}$.  For any open set
$U\subseteq\mathcal C$,
\[
 \Omega_G(J)\cap U\ne\varnothing
 \quad\Longleftrightarrow\quad
 \exists n\in\N\ [A_J p_n\in U].
\]
Hence
\[
 \Omega_G^{-1}(E^+(U))
 =\bigcup_{n=1}^\infty (A_J p_n)^{-1}(U).
\]
Each set in the countable union is $\boldsymbol\Sigma^0_3$ in $\Ideal(G)$,
so $\Omega_G^{-1}(E^+(U))$ is $\boldsymbol\Sigma^0_3$.  By condition~\ref{A2},
every open set in the Wijsman topology on $\Subl(\mathcal C)$ is a countable
union of sets $E^+(U)\setminus E^+(V)$.  Their preimages are countable
unions of intersections of $\boldsymbol\Sigma^0_3$ and $\boldsymbol\Pi^0_3$
sets, which lie in $\boldsymbol\Sigma^0_4$. That proves the assertion that $\Omega_G$ is of $\boldsymbol\Pi^0_3$-measurable.

\end{proof}

\begin{corollary}\label{cor:borel-equivalent-codings}
For $P=\FBL[\ell_1]$ and $\mathcal C=C(\Delta;L_1(0,1))$, the quotient and
sublattice codings of separable real Banach lattices admit Borel
translations in both directions which preserve the coded lattice up
to isometric lattice isomorphism.  In particular, the two isometric
lattice-isomorphism relations are Borel bireducible.
\end{corollary}

\begin{proof}
Combine Theorems~\ref{thm:coding-embedding} and
\ref{thm:reverse-borel-coding}.
\end{proof}

\begin{remark}
The mixed relation
\[
 \{(J,Y)\in\Ideal(P)\times\Subl(\mathcal C):P/J\cong_{\mathrm{lat,iso}}Y\}
\]
is analytic: use the continuous map $Y\mapsto\Psi_{\mathcal C}(Y)$
and the analyticity of the quotient relation proved in
Theorem~\ref{thm:quotient-classification}. Its sections are non-empty,
so the Jankov--von Neumann theorem~\cite{Kechris1995} already gives a
$\sigma(\boldsymbol\Sigma^1_1)$-measurable reverse selector.
Theorem~\ref{thm:reverse-borel-coding} strengthens this to a Borel selector.
\end{remark}

\section{The isometric lattice-isomorphism relations}\label{sec:classification}

\subsection{Classification in the quotient coding}\label{subsec:quotient-classification}

The topological regularity of the coding space does not make its isomorphism
relation Borel.  In fact, already quotients of the form $C(K)$ retain the
full difficulty of classifying compact metrisable spaces up to homeomorphism.
Throughout this subsection, $P=\FBL[\ell_1]$ and
$\mathbb H=[0,1]^{\mathbb N}$.

\begin{theorem}\label{thm:quotient-classification}
There is a continuous embedding
\[
 \kappa:\mathcal F(\mathbb H)\longrightarrow\Ideal(P)
\]
such that $P/\kappa(K)$ is lattice isometric to $C(K)$ for every non-empty
compact $K\subseteq\mathbb H$.  Moreover, for such $K,L$,
\[
 K\cong L
 \quad\Longleftrightarrow\quad
 P/\kappa(K)\cong_{\mathrm{lat,iso}}P/\kappa(L),
\]
where the first relation denotes homeomorphism and the second isometric
lattice isomorphism.  Consequently, isometric lattice isomorphism on
$\Ideal(P)$ is analytic and non-Borel; every orbit equivalence relation
of a Borel action of a Polish group is Borel reducible to this relation.
\end{theorem}

\begin{proof}
Choose a sequence $(h_n)_{n\in\omega}$ dense in the closed unit ball of $C(\mathbb H)$.
The lattice homomorphism
\[
 Q:P\longrightarrow C(\mathbb H),\qquad
 Q\delta_{\ell_1}(e_n)=h_n,
\]
is a metric quotient by Lemma~\ref{lem:metric-quotient}.
For $K\in\mathcal F(\mathbb H)$ let
$R_K:C(\mathbb H)\to C(K)$ be restriction and set
\[
 \kappa(K)=\ker(R_KQ).
\]
By the Tietze extension theorem, $R_K$ is a metric quotient.  A composition
of metric quotient maps is again a metric quotient, so $R_KQ$ induces the
asserted lattice isometry and, for every $z\in P$,
\begin{equation}\label{eq:compact-quotient-distance}
 \dist(z,\kappa(K))=\|R_KQz\|_\infty
 =\max_{t\in K}|(Qz)(t)|.
\end{equation}
The last expression is continuous in $K$ for the Vietoris topology,
since $Qz$ is continuous on the compact space $\mathbb H$. This topology
coincides with the Wijsman topology on $\mathcal F(\mathbb H)$.
Thus $\kappa$ is Wijsman-continuous.  It is injective: if, for example,
$t\in K\setminus L$, choose $h\in C(\mathbb H)$ with $h|_L=0$ and
$h(t)=1$, and lift $h$ through $Q$.  The resulting lift belongs to
$\kappa(L)$ but not to $\kappa(K)$.  Since $\mathcal F(\mathbb H)$ is
compact and $\Ideal(P)$ is Hausdorff, $\kappa$ is a topological embedding.

Homeomorphic compacta have lattice-isometric function spaces by
composition with a homeomorphism.  Conversely, the real Banach--Stone
theorem implies that linearly isometric $C(K)$ and $C(L)$ have
homeomorphic underlying compacta.  This proves the displayed equivalence.
Zielinski's theorem~\cite{Zielinski2016} now gives the assertion about
orbit equivalence relations and, in particular, non-Borelness.

For completeness, we prove analyticity directly.  Fix an enumeration
$(u_n)_{n\in\omega}$ of a countable dense rational vector sublattice of $P$, and write
$\varphi_I(z)=\dist(z,I)$.  Given $I,J\in\Ideal(P)$, code the proposed
images of $u_n+I$ by a sequence $(a_n)\in P^{\mathbb N}$, interpreted
as $(a_n+J)$.  Require the following conditions:
\begin{enumerate}[label=(\roman*)]
\item for every finite rational tuple $(q_1,\ldots,q_k)$ and every
      $(n_1,\ldots,n_k)$,
\[
 \varphi_J\!\left(\sum_{i=1}^kq_i a_{n_i}\right)
 =\varphi_I\!\left(\sum_{i=1}^kq_i u_{n_i}\right);
\]
\item if $u_m=|u_n|$, then $\varphi_J(a_m-|a_n|)=0$;
\item for every $m\in\omega$ and $k\geqslant1$ there exists $n$ such that
      $\varphi_J(u_m-a_n)<1/k$.
\end{enumerate}
These are Borel conditions on $(I,J,(a_n))$: the distance evaluation
$(J,z)\mapsto\varphi_J(z)$ is continuous, and all the quantifiers are
countable.  Condition~(i) makes the assignment a well-defined rational
linear isometry on the dense rational subspace of $P/I$.  It extends to
a real linear isometry $P/I\to P/J$.  Condition~(ii) makes the extension
a lattice homomorphism, and condition~(iii) makes its range dense, hence
surjective because the domain is complete.  Conversely, every isometric
lattice isomorphism admits such a code, by choosing a representative
$a_n\in P$ for each image of $u_n+I$.  The relation in question is
therefore the projection of a Borel subset of
$\Ideal(P)^2\times P^{\mathbb N}$, and is analytic.
\end{proof}

\subsection{Classification in the sublattice coding}

\begin{proposition}\label{prop:sublattice-classification}
The relation of isometric lattice isomorphism on $\Subl(\mathcal C)$ is
analytic and non-Borel. Every orbit equivalence relation of a Borel
action of a Polish group is Borel reducible to this relation.
\end{proposition}

\begin{proof}
Write $x_j=f_j(Y)$ and $z_j=f_j(Z)$ for the continuous dense selections
in two coded lattices. A sequence $(y_j)\in\mathcal C^\omega$ codes
a surjective lattice isometry from $Y$ to $Z$ precisely when:
\begin{enumerate}[label=\textup{(\roman*)}]
\item $y_j\in Z$ for every $j$;
\item $\|\sum_{j<r}q_jy_j\|=\|\sum_{j<r}q_jx_j\|$ for every
$r\geqslant1$ and $q_0,\ldots,q_{r-1}\in\mathbb Q$;
\item $\|y_j-|y_i|\|=\|x_j-|x_i|\|$ for all $i,j$;
\item for every $k\in\omega$ and $m\geqslant1$ there is $j$ with
$\|z_k-y_j\|<1/m$.
\end{enumerate}
These are Borel conditions. Condition~(ii) gives a well-defined
rational-linear isometry on the rational span of the $x_j$, hence a
real-linear isometry $T\colon Y\to Z$ with $Tx_j=y_j$.
To verify preservation of moduli, choose $x_{j_n}\to|x_i|$.
Condition~(iii) gives $y_{j_n}\to|y_i|$, whereas continuity gives
$y_{j_n}\to T|x_i|$. Hence $T|x_i|=|Tx_i|$, and density extends
this identity to every vector. Condition~(iv) gives surjectivity,
because the image of an isometry from a Banach space is closed.
Conversely, a surjective lattice isometry supplies such a sequence.
Projection of this Borel witness set proves analyticity.

For the lower bound, put $\mathbb H=[0,1]^\omega$ and choose a dense
sequence $(h_n)_{n\in\omega}$ in $C(\mathbb H)$. Apply
Lemma~\ref{lem:borel-cantor-surjections} with $M=\mathbb H$ to obtain
Borelly chosen continuous surjections $s_K\colon\Delta\to K$ for
$K\in\mathcal F(\mathbb H)$. Set
\[
 Y_K=\overline{\{(h_n\circ s_K)\mathbf1_{L_1}:n\in\omega\}}
 \subseteq C(\Delta;L_1)=\mathcal C,
\]
where $\mathbf1_{L_1}$ is the constant function on $(0,1)$, of norm one.
For each fixed $n$, the map $s\mapsto(h_n\circ s)\mathbf1_{L_1}$
is continuous in the uniform topology. Thus $K\mapsto Y_K$ is Borel
by Lemma~\ref{lem:closure-map-borel}.

By the Tietze extension theorem, $(h_n|_K)_n$ is dense in $C(K)$.
Consequently the map
\[
 C(K)\longrightarrow\mathcal C,\qquad
 f\longmapsto(f\circ s_K)\mathbf1_{L_1}
\]
is a lattice isometry onto $Y_K$. In particular, $Y_K$ is a closed
sublattice. The Banach--Stone theorem now gives
\[
 K\cong L
 \quad\Longleftrightarrow\quad
 Y_K\cong_{\mathrm{lat,iso}}Y_L.
\]
By Zielinski's theorem~\cite{Zielinski2016}, every orbit equivalence
relation of a Borel action of a Polish group is Borel reducible to
homeomorphism on $\mathcal F(\mathbb H)$. Composition with
$K\mapsto Y_K$ proves the asserted lower bound and non-Borelness.
\end{proof}

\subsection{An upper bound by a Polish-group action}

We use the following general coding theorem of Elliott, Farah, Paulsen,
Rosendal, Toms, and T\"ornquist~\cite[Theorem~2.2]{EFP+2013}.
Let $\mathbb U$ be the Urysohn universal complete separable metric space,
and let $L=(l_j)_{j\in\mathbb N}$ be a countable sequence of positive
integers.  On the standard Borel space
\[
 \mathfrak M(L)=
 \left\{(X,(R_j)):
   \varnothing\ne X\in\mathcal F(\mathbb U),\quad
   R_j\subseteq X^{l_j}\text{ is closed for every }j\right\},
\]
isometric isomorphism of the relational structures
$(X,d,(R_j))$ is Borel reducible to the orbit equivalence relation of the
diagonal action of $\operatorname{Iso}(\mathbb U)$ on $\mathfrak M(L)$.
The domain metric need not be bounded.  In particular, this theorem applies
to the closed graphs of continuous operations on a Banach lattice; no
bounded-sort or continuity-modulus convention is required.

\begin{proposition}\label{prop:orbit-upper}
For every separable Banach lattice $E$, isometric lattice isomorphism on
$\Subl(E)$ is Borel reducible to an orbit equivalence relation of a Borel
Polish-group action.  The same holds for isometric lattice isomorphism
of the quotients coded by $\Ideal(P)$, where $P=\FBL[\ell_1]$.
Consequently both the relation on $\Ideal(P)$ and the relation on
$\Subl(\mathcal C)$ are Borel bireducible with a complete orbit
equivalence relation of Polish-group actions.
\end{proposition}

\begin{proof}
Associate to a real Banach lattice $X$ the relational structure
\[
 \mathscr R(X)=
 \bigl(X,d_X,Z_X,A_X,(S_{q,X})_{q\in\mathbb Q},M_X\bigr),
 \qquad d_X(x,y)=\|x-y\|,
\]
where
\[
\begin{aligned}
 Z_X&=\{0\},&
 A_X&=\{(x,y,x+y):x,y\in X\},\\
 S_{q,X}&=\{(x,qx):x\in X\},&
 M_X&=\{(x,|x|):x\in X\}.
\end{aligned}
\]
All the relations are closed and their arities form one fixed countable
signature.  An isometric isomorphism between these structures preserves
zero, addition, rational scalar multiplication, and modulus.  Continuity
extends rational linearity to real linearity.  Thus such isomorphisms are
exactly the isometric lattice isomorphisms.

We verify explicitly that both coding spaces map Borelly to
$\mathfrak M(L)$.  For sublattice codes, fix an isometric embedding
$\iota:E\to\mathbb U$ and continuous dense selections $f_n(Y)\in Y$.
The closed domain $\iota(Y)$ is the closure of
$\{\iota(f_n(Y)):n\in\mathbb N\}$.  The addition relation is the closure
of
\[
 \{(\iota(f_n(Y)),\iota(f_m(Y)),
       \iota(f_n(Y)+f_m(Y))):n,m\in\mathbb N\},
\]
and the scalar and modulus relations are obtained in the same way from
$\iota(qf_n(Y))$ and $\iota(|f_n(Y)|)$.  These are Borel hyperspace maps
by the countable-closure lemma.  They are precisely the graphs of the
transported operations, because the selections are dense and the
operations are continuous.  The zero relation is $\{\iota(0)\}$.
This gives the required Borel coding of $\mathscr R(Y)$.

For quotient codes, compose the Borel map $\Omega_P$ from
Theorem~\ref{thm:reverse-borel-coding} with the sublattice coding just
constructed for $E=\mathcal C$. This codes each $P/I$ by an isometrically
isomorphic relational structure and proves the required Borel dependence.

Theorem~2.2 of~\cite{EFP+2013} now gives the two asserted orbit upper
bounds.  By Theorem~\ref{thm:quotient-classification} and
Proposition~\ref{prop:sublattice-classification}, every orbit equivalence
relation of a Borel Polish-group action already reduces to each of the
relations on $\Ideal(P)$ and $\Subl(\mathcal C)$.  The upper and lower
bounds together give their Borel bireducibility with the complete orbit
equivalence relation.
\end{proof}

\section{Lattice geometry and further questions}\label{sec:further}

We study geometric properties through continuous dense selections and
quotient-norm coordinates. Besides the uniform and non-uniform properties
of Kurc, we give explicit tests for classical lattice classes, atom counts,
compact order intervals, and order continuity.

\subsection{Principal ideals and SICK compacta}\label{subsec:principal-ideals}

Let $E$ be a Banach lattice and $u \in E$ a non-zero positive element.  The
principal ideal $E_u$ need not be closed in the ambient norm of~$E$ (for
instance, taking $E = L_1[0,1]$ and $u = \mathbf{1}$ gives $E_u = L_\infty$,
which is not norm-closed in~$L_1$).  However, equipped with the \emph{gauge norm}
\[
    \lVert x \rVert_u := \inf\{r > 0 \colon \lvert x \rvert \leqslant r u\},
\]
the ideal $E_u$ becomes an AM-space with unit~$u$, hence by the
Krein--Kakutani representation theorem is lattice isometric to $C(K; \R)$ for
some compact Hausdorff space~$K$.  When the ambient norm-closure is needed we
write $\overline{E_u}^{\lVert\cdot\rVert_E}$.

We denote this compact space, determined up to homeomorphism, by $K_u(E)$.
Equivalently, it is the space of unital lattice homomorphisms
$(E_u,\|\cdot\|_u)\to\R$, with pointwise convergence. This formulation
uses the principal ideal with its gauge norm and does not assert that it
is a closed subspace of $E$ in the original norm.

\begin{definition}
A compact Hausdorff space $K$ is a \emph{SICK compactum} (Separable
Ideal $C(K)$) if $K$ is homeomorphic to $K_u(E)$ for some separable Banach
lattice $E$ and some non-zero $u \in E^+$.
\end{definition}

The general, potentially non-metrisable class was introduced and studied in
\cite{AMRT2022Sick}. Its metrisable part has an immediate description.

\begin{proposition}\label{prop:metric-sick}
Every non-empty compact metrisable space is a SICK compactum. Consequently,
with $\mathbb H=[0,1]^\omega$,
\[
 \mathrm{SICK}_{\mathrm{met}}=\mathcal F(\mathbb H).
\]
\end{proposition}

\begin{proof}
Given a non-empty compact metrisable space $K$, the Banach lattice $E=C(K)$
is separable. For $u=\mathbf1$, the principal ideal $E_u$ is all of $C(K)$
and its gauge norm is the supremum norm. Hence $K_u(E)$ is homeomorphic to $K$.
\end{proof}

Thus the substantive SICK questions concern non-metrisable structure spaces,
not membership in the hyperspace of compact metrisable spaces. For example,
\cite{AMRT2022Sick} shows that every non-metrisable SICK compactum contains
a copy of $\beta\N$.

\subsection{Monotonicity and order smoothness in the sense of Kurc}\label{subsec:kurc}

We now turn to several lattice-geometric properties introduced by
Kurc~\cite{Kurc1993Dual} and determine their descriptive complexity within
$\Subl(E)$.

\begin{definition}[Kurc~\cite{Kurc1993Dual}]\label{def:kurc}
Let $X$ be a Banach lattice.
\begin{enumerate}[label=(\roman*)]
    \item $X$ is \emph{strictly monotone} (STM) if $\lVert x - y \rVert < \lVert x \rVert$
      whenever $\mathbf{0}_X < y \leqslant x$.
    \item For $\varepsilon \in (0, 1]$ define
      \[
          \delta_X(\varepsilon)
          = \inf\bigl\{1 - \lVert x - y \rVert \colon
            \mathbf{0}_X \leqslant y \leqslant x,\;
            \lVert x \rVert = 1,\;
            \lVert y \rVert \geqslant \varepsilon\bigr\}.
      \]
      Then $X$ is \emph{uniformly monotone} (UM) if $\delta_X(\varepsilon) > 0$
      for all $\varepsilon \in (0, 1]$.
    \item $X$ is \emph{order smooth}, in abbreviation \emph{(o)-Sm}, if for
      every $x$ in the positive part of the unit sphere $S(X_+)$, the set
      $\partial_+\lVert x \rVert
      = \{x^* \in S(X^*_+) \colon \langle x, x^* \rangle = \lVert x \rVert\}$
      contains no non-degenerate order interval $[a,b]$ with $a<b$.
    \item For $\tau \in [0, 1]$ define the \emph{order modulus of smoothness}
      \[
          \rho_X(\tau)
          = \sup\bigl\{\lVert x \vee \tau y \rVert - 1 \colon
            \mathbf{0}_X \leqslant x, y,\;
            \lVert x \rVert = \lVert y \rVert = 1\bigr\}.
      \]
      Then $X$ is \emph{order uniformly smooth}, in abbreviation
      \emph{(o)-USm}, if $\rho_X(\tau) / \tau \to 0$ as $\tau \downarrow 0$.
\end{enumerate}
\end{definition}

For $X\ne\{0\}$ put
\[
 \sigma_X(\varepsilon)=\inf\{\|x+y\|-1:
 x,y\in X_+,\ \|x\|=1,\ \|y\|\geqslant\varepsilon\}.
\]
Scaling the second vector down and using monotonicity of the norm gives
\begin{equation}\label{eq:sigma-unit-reduction}
 \sigma_X(\varepsilon)=\inf_{x,y\in S(X_+)}
       (\|x+\varepsilon y\|-1).
\end{equation}
In particular $0\leqslant\sigma_X(\varepsilon)\leqslant\varepsilon$.
For the zero lattice we set
$\sigma_{\{0\}}(\varepsilon)=\varepsilon$ and $\rho_{\{0\}}(\tau)=0$;
thus the zero lattice satisfies both uniform properties by convention,
consistently with their vacuous formulations.

The following elementary estimates give the comparison we need between the
two monotonicity moduli. An exact identity without a change of parameter
would be incorrect: on $\R$, both $\delta_\R(\varepsilon)$ and
$\sigma_\R(\varepsilon)$ equal $\varepsilon$.

\begin{lemma}\label{lem:moduli-comparison}
For every non-zero Banach lattice $X$ and $0<\varepsilon\leqslant1$,
\begin{equation}\label{eq:moduli-comparison}
 \delta_X(\varepsilon)\geqslant
 \frac{\sigma_X(\varepsilon)}{1+\sigma_X(\varepsilon)},
 \qquad
 \sigma_X(\varepsilon)\geqslant
 \frac{\delta_X(\varepsilon/(1+\varepsilon))}
 {1-\delta_X(\varepsilon/(1+\varepsilon))}.
\end{equation}
Consequently, $X$ is uniformly monotone if and only if
$\sigma_X(\varepsilon)>0$ for every $0<\varepsilon\leqslant1$.
\end{lemma}

\begin{proof}
Suppose $0\leqslant y\leqslant x$, $\|x\|=1$ and $\|y\|\geqslant\varepsilon$.
Set $a=\|x-y\|$. If $a>0$, then
\[
 a^{-1}=\|(x-y)/a+y/a\|\geqslant1+\sigma_X(\varepsilon),
\]
since $a\leqslant1$ and $\|y/a\|\geqslant\varepsilon$.
This proves the first estimate; the case $a=0$ is immediate.
For the second, take $v,w\in S(X_+)$ and set
$r=\|v+\varepsilon w\|\leqslant1+\varepsilon$.
The pair $x=(v+\varepsilon w)/r$, $y=\varepsilon w/r$ is admissible for
$\delta_X(\varepsilon/(1+\varepsilon))$, whence
\[
 \delta_X(\varepsilon/(1+\varepsilon))\leqslant1-r^{-1}.
\]
Rearrange and take the infimum over $v,w$. The denominator is positive because
$\delta_X(t)\leqslant t<1$ at $t=\varepsilon/(1+\varepsilon)$,
as is seen by taking $y=tx$ with $x\in S(X_+)$.
\end{proof}

For context, the duality theorem of Kurc~\cite[Theorem~5(b)]{Kurc1993Dual}
states that $X$ is uniformly monotone if and only if $X^*$ is order uniformly
smooth. Our complexity argument uses the direct formulae above and does not
require a coding of the dual as a separable Banach lattice.

\begin{theorem}\label{thm:kurc-borel}
Let $E$ be a separable Banach lattice. Endow $\Subl(E)$ with an admissible
topology satisfying~\ref{A3}, in particular the Wijsman topology.
\begin{enumerate}[label=\textup{(\alph*)}]
\item\label{kurc:baire1}
For every fixed $\varepsilon,\tau\in(0,1]$, the map
$F\mapsto\sigma_F(\varepsilon)$ is upper semicontinuous and
$F\mapsto\rho_F(\tau)$ is lower semicontinuous. In particular both are
of Baire class~$1$.
\item\label{kurc:UM}
The class $\mathrm{UM}(E)$ is $\boldsymbol\Pi^0_3$ in $\Subl(E)$.
\item\label{kurc:oUSm}
The class $\mathrm{(o)\text{-}USm}(E)$ is $\boldsymbol\Pi^0_3$ in $\Subl(E)$.
\item\label{kurc:STM}
The class $\mathrm{STM}(E)$ is coanalytic.
\item\label{kurc:oSm}
The class $\mathrm{(o)\text{-}Sm}(E)$ is coanalytic.
\end{enumerate}
\end{theorem}

\begin{proof}
Choose continuous dense selections $f_j$ as in
Theorem~\ref{thm:selection}. On the open set
$U_j=\{F:f_j(F)\ne0\}$ the map
\[
 s_j(F)=\frac{|f_j(F)|}{\|f_j(F)\|}
\]
is continuous. For every non-zero $F$, the vectors $s_j(F)$ with
$F\in U_j$ are dense in $S(F_+)$: approximate any positive unit vector
by the dense sequence $f_j(F)$ and then take moduli and normalise.

For~\ref{kurc:baire1}, let $c\in\R$. On non-zero lattices,
\[
 \sigma_F(\varepsilon)<c
 \quad\Longleftrightarrow\quad
 \exists i,j\ [F\in U_i\cap U_j,\
       \|s_i(F)+\varepsilon s_j(F)\|-1<c].
\]
The right-hand side defines an open set. If $c\leqslant\varepsilon$, the
zero lattice does not belong to this sublevel set, whereas if
$c>\varepsilon$ the sublevel set is all of $\Subl(E)$.
This proves upper semicontinuity globally. Likewise, for non-zero $F$,
\[
 \rho_F(\tau)>c
 \quad\Longleftrightarrow\quad
 \exists i,j\ [F\in U_i\cap U_j,\
       \|s_i(F)\vee\tau s_j(F)\|-1>c].
\]
This is open by continuity of the lattice operations. For $c\geqslant0$
the zero lattice is excluded, and for $c<0$ the superlevel set is the whole
space. Hence $F\mapsto\rho_F(\tau)$ is lower semicontinuous.

For~\ref{kurc:UM}, the function $\varepsilon\mapsto\sigma_F(\varepsilon)$
is non-decreasing directly from its definition. By
Lemma~\ref{lem:moduli-comparison},
\[
 \mathrm{UM}(E)=\bigcap_{n\geqslant1}\bigcup_{m\geqslant1}
 \{F:\sigma_F(1/n)\geqslant1/m\}.
\]
The innermost sets are closed by upper semicontinuity, proving the bound.

For~\ref{kurc:oUSm}, the function $t\mapsto\|x\vee ty\|$ is convex for
$x,y\in S(F_+)$. Indeed, its value in the lattice order at a convex
combination is at most the corresponding convex combination of
$x\vee sy$ and $x\vee ty$, and the norm is monotone on positive elements.
Its value at zero is $1$, so
$(\|x\vee ty\|-1)/t$ is non-decreasing for $t>0$.
Taking suprema shows that $\rho_F(t)/t$ is also non-decreasing. Therefore
\[
 \mathrm{(o)\text{-}USm}(E)
 =\bigcap_{k\geqslant1}\bigcup_{n\geqslant1}
   \{F:\rho_F(1/n)<1/(kn)\}.
\]
For a lower semicontinuous function $h$, each strict sublevel set
$\{h<c\}=\bigcup_{m\geqslant1}\{h\leqslant c-1/m\}$ is $F_\sigma$.
The displayed expression consequently defines a $\boldsymbol\Pi^0_3$ set.

For~\ref{kurc:STM}, failure is witnessed by $x,y\in F$ satisfying
\[
 0< y\leqslant x,\qquad \|x-y\|=\|x\|.
\]
These conditions define a Borel subset of $\Subl(E)\times E^2$:
membership is closed by~\ref{A3}, positivity and the norm equality are
closed, and $y\ne0$ is open. Its projection is analytic.

Finally, for~\ref{kurc:oSm}, equip $B_{E^*}$ with its compact metrisable
weak* topology. Write $p_j(F)=|f_j(F)|$, a dense sequence in $F_+$.
Failure of order smoothness is equivalent to the existence of
$x\in F_+$ and $a,b\in B_{E^*}$ such that
\begin{align*}
 &\|x\|=1,\qquad a(x)=b(x)=1,\\
 &a(p_j(F))\geqslant0,\qquad (b-a)(p_j(F))\geqslant0
       &&(j\in\omega),\\
 &(b-a)(p_j(F))>0 &&\text{for some }j\in\omega.
\end{align*}
In fact, these conditions say that $a|_F$ and $b|_F$ are distinct comparable
positive norming functionals at $x$. Every functional between them has norm
one and value one at $x$, so their entire non-degenerate order interval lies
in $\partial_+\|x\|$. Conversely, the endpoints of any such interval admit
norm-preserving Hahn--Banach extensions to $E$, which supply $a,b$.
No positivity of these extensions on the whole ambient lattice is required.

Evaluation on $B_{E^*}\times E$ is jointly continuous for the weak* and norm
topologies: the uniform bound on the functionals controls variation of the
vector. Thus all the displayed conditions, together with $x\in F_+$,
define a Borel subset of
$\Subl(E)\times E\times B_{E^*}^2$.
Its projection is the complement of $\mathrm{(o)\text{-}Sm}(E)$, proving
coanalyticity.
\end{proof}

\begin{remark}\label{rem:kurc-topology-transfer}
The $\boldsymbol\Pi^0_3$ bounds just proved hold in every admissible topology
satisfying~\ref{A3}: the same argument applies separately in each topology.
Thus no increase in Borel rank is needed for these particular classes.
Analytic and coanalytic status depends only on the common Effros--Borel
structure.
\end{remark}

\begin{proposition}\label{prop:quotient-uniform-upper}
Let $P$ be a separable Banach lattice and equip $\Ideal(P)$ with its
norm-Wijsman topology. For fixed $\varepsilon,\tau\in(0,1]$, the maps
\[
 I\longmapsto\sigma_{P/I}(\varepsilon),\qquad
 I\longmapsto\rho_{P/I}(\tau)
\]
are upper and lower semicontinuous, respectively, with the zero-quotient
conventions used above. Consequently, the classes of ideals $I$ for which
$P/I$ is uniformly monotone or order uniformly smooth are
$\boldsymbol\Pi^0_3$.
\end{proposition}

\begin{proof}
Fix a dense sequence $(p_i)$ in $P$, let $q_I:P\to P/I$ be the quotient map,
and put
\[
 t_i(I)=\operatorname{dist}(|p_i|,I),\qquad U_i=\{I:t_i(I)>0\}.
\]
The functions $t_i$ are continuous and the sets $U_i$ are open. For $I\ne P$,
the vectors $q_I|p_i|/t_i(I)$ with $I\in U_i$ are dense in the positive
unit sphere of $P/I$. Indeed, every positive quotient vector has a positive
lift, and the vectors $|p_i|$ are dense in $P_+$; taking quotient images and
normalising proves the assertion.

For $I\in U_i\cap U_j$, define
\begin{align*}
 A_{ij}(I)&=\operatorname{dist}\left(
       \frac{|p_i|}{t_i(I)}+\varepsilon\frac{|p_j|}{t_j(I)},I\right)-1,\\
 B_{ij}(I)&=\operatorname{dist}\left(
       \frac{|p_i|}{t_i(I)}\vee\tau\frac{|p_j|}{t_j(I)},I\right)-1.
\end{align*}
Both functions are continuous on their open domains: the lifted vectors
vary continuously in $P$, and $(x,I)\mapsto\operatorname{dist}(x,I)$ is
jointly continuous. Since quotient maps are lattice homomorphisms, for $I\ne P$ density
gives
\[
 \sigma_{P/I}(\varepsilon)=\inf_{i,j:\ I\in U_i\cap U_j} A_{ij}(I),
 \qquad
 \rho_{P/I}(\tau)=\sup_{i,j:\ I\in U_i\cap U_j} B_{ij}(I).
\]
Thus strict sublevel sets of the first function and strict superlevel sets
of the second are open away from the zero quotient. The conventions
$\sigma_{\{0\}}(\varepsilon)=\varepsilon$ and $\rho_{\{0\}}(\tau)=0$ make
the same assertion global, since
$0\leqslant\sigma_{P/I}(\varepsilon)\leqslant\varepsilon$ and
$\rho_{P/I}(\tau)\geqslant0$. The countable modulus tests in the proof of
Theorem~\ref{thm:kurc-borel} now give both
$\boldsymbol\Pi^0_3$ upper bounds.
\end{proof}

\subsection{Sharpness of the uniform bounds}\label{subsec:uniform-sharpness}

We now prove that the two uniform upper bounds are sharp in the universal
sublattice space. The construction uses two-dimensional lattice norms with
flat portions confined to increasingly small scales. Plain $\ell_p^2$
blocks would not suffice: as $p\to\infty$, they lose uniform monotonicity
at every fixed positive scale.

\begin{lemma}\label{lem:two-choice-realisation}
For each $i\in\N$, let $V_i^0,V_i^1$ be two lattice norms on the
coordinatewise ordered space $\R^2$. There is a continuous map
\[
 \Phi:\{0,1\}^{\N}\longrightarrow\Subl(\mathcal C)
\]
such that $\Phi(a)$ is lattice isometric to
$\bigl(\bigoplus_i V_i^{a_i}\bigr)_{\ell_2}$.
\end{lemma}

\begin{proof}
Put $U=C([0,1];\ell_1^2)$. Every lattice norm $q$ on $\R^2$ has an
isometric lattice embedding into $U$. Indeed, writing $q^*$ for its dual
norm, the formula
\[
 L_q(x_1,x_2)(t)
 =\frac{(t x_1,(1-t)x_2)}{q^*(t,1-t)}\qquad(0\leqslant t\leqslant1)
\]
defines such an embedding: its norm is the supremum of the pairings of
$(|x_1|,|x_2|)$ with positive dual unit vectors. Let $L_i^b$ be its range
for $V_i^b$, and put
\[
 E=\left(\bigoplus_i U\right)_{\ell_2},\qquad
 F_a=\left(\bigoplus_i L_i^{a_i}\right)_{\ell_2}\subseteq E.
\]
Fix an isometric lattice embedding $T:E\to\mathcal C$ using
\cite{LLOT2019Uni}, and set $\Phi(a)=T(F_a)$.

To check the required norm-Wijsman continuity in the actual ambient space
$\mathcal C$, put
\[
 H_i=L_i^0+L_i^1,\qquad
 H=\left(\bigoplus_i H_i\right)_{\ell_2}\subseteq E.
\]
Each $H_i$ has dimension at most four. Therefore $H$ is reflexive, and all
$F_a$ lie in this fixed reflexive subspace. Suppose $a^{(k)}\to a$ and fix
$w\in\mathcal C$. Every finitely supported vector of $F_a$ belongs to
$F_{a^{(k)}}$ for all sufficiently large $k$, because each of its finitely
many coordinates eventually stabilises. Density of those vectors gives
\[
 \limsup_k\operatorname{dist}(w,\Phi(a^{(k)}))
 \leqslant\operatorname{dist}(w,\Phi(a)).
\]
For the reverse inequality, pass to a subsequence realising the lower
limit and choose bounded approximate nearest points
$z_k\in\Phi(a^{(k)})$. Reflexivity of $T(H)$ gives a weakly convergent
subsequence, with limit $z\in T(H)$. Every coordinate of $T^{-1}z$ belongs
to the corresponding $L_i^{a_i}$, since the coordinate spaces eventually
stabilise and are weakly closed. Thus $z\in\Phi(a)$. Weak lower
semicontinuity of the norm gives the reverse inequality. This proves
continuity of every Wijsman distance coordinate.
\end{proof}

\begin{lemma}\label{lem:sum-uniform-monotonicity}
Suppose $(X_i)$ is a family of non-zero Banach lattices such that
\[
 \inf_i\delta_{X_i}(\varepsilon)>0
 \quad\text{for every }0<\varepsilon\leqslant1.
\]
Then $X=(\bigoplus_iX_i)_{\ell_2}$ is uniformly monotone.
\end{lemma}

\begin{proof}
Take $0\leqslant y\leqslant x$ with $\|x\|=1$ and $\|y\|\geqslant\varepsilon$.
Write $a_i=\|x_i\|$, $b_i=\|y_i\|$, put $t=\varepsilon/2$, and let
$A=\{i:b_i\geqslant t a_i\}$. Since $b_i\leqslant a_i$,
\[
 \sum_{i\in A}a_i^2\geqslant\sum_{i\in A}b_i^2
 \geqslant\varepsilon^2-t^2=3\varepsilon^2/4.
\]
For $d=\inf_i\delta_{X_i}(t)>0$, the definition of the moduli gives
$\|x_i-y_i\|\leqslant(1-d)a_i$ on $A$ and
$\|x_i-y_i\|\leqslant a_i$ elsewhere. Consequently,
\[
 \|x-y\|^2\leqslant1-(2d-d^2)3\varepsilon^2/4<1.
\]
The bound depends only on $\varepsilon$, proving uniform monotonicity.
\end{proof}

\begin{theorem}\label{thm:uniform-complete}
The classes of uniformly monotone and of order uniformly smooth sublattices
of $\mathcal C$ are $\boldsymbol\Pi^0_3$-complete for the norm-Wijsman
topology. The corresponding classes of quotient codes in
$\Ideal(\FBL[\ell_1])$ are also $\boldsymbol\Pi^0_3$-complete.
\end{theorem}

\begin{proof}
We prove completeness in $\Subl(\mathcal C)$. Consider the standard
$\boldsymbol\Pi^0_3$-complete set
\[
 \mathscr P_3=\{a\in\{0,1\}^{\N\times\N}:
       \forall m\ \exists N\ \forall n\geqslant N\quad a_{m,n}=0\};
\]
see~\cite{Kechris1995}. Index $m,n$ starting at one and put
$r_m=2^{-m}$. For $0\leqslant r\leqslant1/2$ and $2\leqslant p<\infty$,
define the lattice norm
\[
 N_{r,p}(x,y)=\left(|x|^p+|y|^p+
    \left(\frac{\sqrt{x^2+y^2}}{\sqrt{1+r^2}}\right)^p\right)^{1/p}.
\]
At $p=\infty$, use the maximum of the three displayed terms. These norms
vary continuously, uniformly on bounded subsets of $\R^2$, in the compact
parameters $(r,1/p)\in[0,1/2]\times[0,1/2]$. For finite $p$ they are
strictly monotone, because the Euclidean term strictly increases whenever
a positive vector is strictly increased. At $r=0,p=\infty$ the norm is
Euclidean. Thus every norm with finite $p$, and every norm with $r=0$,
is uniformly monotone, by finite-dimensional compactness.

Set
\[
 p_{m,n}(a)=\begin{cases}2,&a_{m,n}=0,\\n+2,&a_{m,n}=1,\end{cases}
 \qquad
 X_a=\left(\bigoplus_{m,n}(\R^2,N_{r_m,p_{m,n}(a)})\right)_{\ell_2}.
\]
If $a\in\mathscr P_3$, the closure $K_a$ of its parameter set
$\{(r_m,1/p_{m,n}(a)):m,n\in\N\}$ contains no point $(r,0)$ with $r>0$:
for fixed $m$ only finitely many $n$ use the second choice, and the only
accumulation point of the sequence $(r_m)$ is zero. Hence every norm
parametrised by $K_a$ is uniformly monotone. For each fixed
$\varepsilon>0$, its additive modulus $\sigma_N(\varepsilon)$ depends
continuously on the norm parameter. To see this directly, let $S_2^+$ be
the positive part of the Euclidean unit circle and write
\[
 \sigma_N(\varepsilon)=
 \min_{u,v\in S_2^+}
 \left[N\left(\frac{u}{N(u)}+
                  \varepsilon\frac{v}{N(v)}\right)-1\right].
\]
The minimand is jointly continuous on compact parameter sets. Positivity
and compactness of $K_a$ therefore give a positive common lower bound for
$\sigma_N(\varepsilon)$, and
Lemma~\ref{lem:moduli-comparison} gives a positive common lower bound for
$\delta_N(\varepsilon)$. Lemma~\ref{lem:sum-uniform-monotonicity} shows that
$X_a$ is uniformly monotone.

If $a\notin\mathscr P_3$, fix a row $m$ containing infinitely many ones. Along its
corresponding blocks, $p=n+2\to\infty$. For $r=r_m$,
\[
 N_{r,\infty}(1,0)=N_{r,\infty}(1,r)=1,
 \qquad N_{r,\infty}(0,r)=r.
\]
Normalising $(1,r)$ in the finite-$p$ blocks, with the dominated vector
$(0,r)$ normalised by the same factor, shows that
$\delta_{X_a}(r/2)=0$. Thus $X_a$ is not uniformly monotone.
We have proved
\[
 a\in\mathscr P_3\quad\Longleftrightarrow\quad X_a\text{ is uniformly monotone}.
\]
Lemma~\ref{lem:two-choice-realisation}, with $i=(m,n)$, supplies a continuous
realisation of the $X_a$ in $\Subl(\mathcal C)$. Together with
Theorem~\ref{thm:kurc-borel}, this proves the first completeness assertion.

For order uniform smoothness, apply the same realisation lemma to the dual
norms $N_{r_m,2}^*$ and $N_{r_m,n+2}^*$. The resulting sublattice is lattice
isometric to
\[
 Y_a=\left(\bigoplus_{m,n}
       (\R^2,N_{r_m,p_{m,n}(a)}^*)\right)_{\ell_2}=X_a^*.
\]
Kurc's duality theorem~\cite[Theorem~5(b)]{Kurc1993Dual} gives
\[
 a\in\mathscr P_3\quad\Longleftrightarrow\quad Y_a
       \text{ is order uniformly smooth}.
\]
This is a second continuous reduction, and proves the other completeness
assertion. Finally, the continuous isometry-preserving bridge to ideal
codes transfers both lower bounds to $\Ideal(\FBL[\ell_1])$; completeness
there follows from Proposition~\ref{prop:quotient-uniform-upper}.
\end{proof}

\subsection{Strict monotonicity: a test family and the completeness problem}\label{subsec:STM-complete}

Theorem~\ref{thm:kurc-borel} gives the coanalytic upper bound for strict monotonicity.
A natural next question is whether this upper bound is sharp.  The following
simple observation is a useful source of non-strictly monotone examples, but it
also shows why a direct reduction through compact hyperspaces is too naive.

\begin{lemma}\label{lem:C(K)-STM}
Let $K$ be a non-empty compact Hausdorff space.  Then the Banach lattice $C(K)$
(with the supremum norm and pointwise order) is strictly monotone if and only if
$K$ is a singleton.
\end{lemma}

\begin{proof}
If $K$ is a singleton, then $C(K)\cong\R$, which is strictly monotone.

Conversely, assume that $K$ contains two distinct points $p\ne q$.  By
Urysohn's lemma there exists $y\in C(K)$ with $0\leqslant y\leqslant 1$, $y(p)=0$, and
$y(q)=1$.  Let $x\equiv 1$.  Then $0< y\leqslant x$, but $x-y$ attains the
value $1$ at~$p$, so
\[
    \|x-y\|_\infty=1=\|x\|_\infty .
\]
Thus $C(K)$ is not strictly monotone.
\end{proof}

\begin{remark}\label{rem:STM-complete-caution}
Lemma~\ref{lem:C(K)-STM} suggests trying to reduce well-foundedness of trees to
strict monotonicity by assigning to a tree $T$ a compact space $K_T$ which is a
singleton exactly when $T$ is well founded.  Such an argument cannot proceed via
a Borel map $T\mapsto K_T$ into the hyperspace $\mathcal F(\Delta)$: the set of
singleton compacta is Borel in $\mathcal F(\Delta)$, whereas the classical
well-founded-tree set is non-Borel.  One must also use trees on
$\N^{<\omega}$, not on $2^{<\omega}$, for the classical
$\boldsymbol{\Pi}^1_1$-complete well-foundedness problem; on finitely branching
binary trees, well-foundedness is Borel by K\"onig's lemma.

Consequently, a proof that $\mathrm{STM}(\mathcal C)$ is
$\boldsymbol{\Pi}^1_1$-hard would need a different Banach-lattice
coding, rather than only the elementary compact-space test above.
\end{remark}

\begin{problem}\label{prob:STM-complete}
Is $\mathrm{STM}(\mathcal C)\subseteq\Subl(\mathcal C)$
$\boldsymbol{\Pi}^1_1$-complete?  More generally, for which separable ambient
Banach lattices $E$ is $\mathrm{STM}(E)$ coanalytic-complete in $\Subl(E)$?
\end{problem}

\subsection{Classical lattice classes and atom counts}\label{subsec:classical-classes}

Fix $1\leqslant p<\infty$. A Banach lattice $X$ is an \emph{abstract
$L_p$-space} if
\[
 \|u+v\|^p=\|u\|^p+\|v\|^p
 \qquad(u,v\in X_+,\ u\wedge v=0).
\]
It is an \emph{AM-space} if
$\|u\vee v\|=\max\{\|u\|,\|v\|\}$ for all $u,v\in X_+$.
Both definitions include the zero lattice. We denote the corresponding
classes in $\Subl(E)$ by $\mathscr L_p(E)$ and $\mathscr M(E)$.
An atom is a non-zero positive vector $a$ for which every vector in $[0,a]$
is a scalar multiple of $a$. When counting atoms, proportional positive
atoms are counted only once, or equivalently each atom is normalised to
have norm one.

The Bohnenblust--Kakutani--Nakano representation theorem identifies
abstract $L_p$-spaces with the usual
$L_p(\mu)$ lattices. In the separable case the atomic band is lattice
isometric to $\ell_p(A)$ for a finite or countable set $A$, and a non-zero
atomless band is lattice isometric to $L_p[0,1]$. We use these standard
facts in their lattice-isometric form; see
\cite[Section~2, Fact~3.5 and Remark~3.16]{BBH2011}.
Thus a separable abstract $L_p$-space has the form
\begin{equation}\label{eq:separable-lp-types}
 \ell_p(A)\oplus_p H,\qquad
 H\in\{\{0\},L_p[0,1]\},\quad
 |A|\in\{0,1,2,\ldots,\aleph_0\}.
\end{equation}
In particular, the number of atoms alone does not decide its type: one must
also record whether the atomless band is zero.

\begin{proposition}\label{prop:closed-lp-am}
For every separable Banach lattice $E$, the classes $\mathscr L_p(E)$ and
$\mathscr M(E)$ are closed in $\Subl(E)$ with the norm-Wijsman topology.
\end{proposition}

\begin{proof}
Let $f_j$ be continuous dense selections from
Theorem~\ref{thm:selection}, and put $a_j(F)=|f_j(F)|$.
For every $i,j$ the vectors
\[
 d_{ij}(F)=(a_i(F)-a_j(F))^+,
 \qquad d_{ji}(F)=(a_j(F)-a_i(F))^+
\]
are disjoint, positive, and depend continuously on $F$.
These pairs are dense in the pairs of disjoint positive vectors of $F$:
if $u\wedge v=0$, approximate $u,v$ by $a_i(F),a_j(F)$ and use
$(u-v)^+=u$, $(v-u)^+=v$.
Consequently $F\in\mathscr L_p(E)$ if and only if
\[
 \|d_{ij}(F)+d_{ji}(F)\|^p
   =\|d_{ij}(F)\|^p+\|d_{ji}(F)\|^p
 \qquad(i,j\in\omega).
\]
This is a countable intersection of closed conditions. Likewise,
$F\in\mathscr M(E)$ if and only if
\[
 \|a_i(F)\vee a_j(F)\|
       =\max\{\|a_i(F)\|,\|a_j(F)\|\}
 \qquad(i,j\in\omega),
\]
again a countable collection of closed conditions.
\end{proof}

\begin{theorem}\label{thm:lp-atom-counts}
Fix a separable Banach lattice $E$ and $1\leqslant p<\infty$.
For each $n\geqslant1$ the set
\[
 \mathscr A_{p,\geqslant n}(E)
   =\{F\in\mathscr L_p(E):F\text{ has at least }n\text{ atoms}\}
\]
is $F_\sigma$ in $\Subl(E)$. Consequently, within $\mathscr L_p(E)$,
being atomless is $G_\delta$, having exactly $n$ atoms is
$\boldsymbol\Delta^0_3$, and having infinitely many atoms is
$\boldsymbol\Pi^0_3$.
The class of lattices in $\Subl(E)$ lattice isometric to $L_p[0,1]$ is
$G_\delta$.
\end{theorem}

\begin{proof}
Put $c=2^{-1/p}$, and for $x\in F_+$ define
\begin{equation}\label{eq:balanced-defect}
 b_F(x)=\inf_{\substack{u,v\in F_+\\u\wedge v=0}}
 \max\bigl\{\|x-u-v\|,
             |\|u\|-c\|x\||,
             |\|v\|-c\|x\||\bigr\}.
\end{equation}
For fixed $F$, this is a $1$-Lipschitz function of $x$: the three entries
inside the maximum change by at most $\|x-x'\|$ when $x$ is replaced by
$x'$, and taking an infimum preserves this estimate.
By the dense disjoint pairs in the previous proof, the infimum can be taken
over $u=d_{rs}(F)$, $v=d_{sr}(F)$.
It follows that, whenever $x(F)\in F_+$ is a continuous selection,
$F\mapsto b_F(x(F))$ is upper semicontinuous, being the infimum of countably
many continuous functions.

We need two facts about~\eqref{eq:balanced-defect}. First, if $a$ is a
positive unit atom, then
\begin{equation}\label{eq:atom-defect-lower}
 b_F(a)\geqslant c/2.
\end{equation}
Indeed, for disjoint positive $u,v$, at least one, say $v$, is disjoint
from $a$. Hence $v$ is disjoint from $a-u$, so
$\|a-u-v\|\geqslant\|v\|$. If the maximum in
\eqref{eq:balanced-defect} were less than $c/2$, it would simultaneously
imply $\|v\|>c/2$ and $\|a-u-v\|<c/2$, a contradiction.
Second, suppose $F\in\mathscr L_p(E)$. If the band generated by $x$
contains no atoms, then $b_F(x)=0$.
In an $L_p(\mu)$ representation, split the atomless finite measure
$|x|^p\,d\mu$ into two equal halves, and let $u,v$ be the restrictions of
$x$ to these halves. Then $u+v=x$, $u\wedge v=0$, and
$\|u\|=\|v\|=c\|x\|$. The case $x=0$ is immediate.

For an index tuple $\mathbf j=(j_1,\ldots,j_n)$ define
\[
 v_i^{\mathbf j}(F)
    =\left(a_{j_i}(F)-\bigvee_{k\ne i}a_{j_k}(F)\right)^+
       \qquad(1\leqslant i\leqslant n),
\]
where the empty join is zero. These are continuous selections of disjoint
positive $n$-tuples, and their values are dense among all such tuples.
To see density, approximate any disjoint $x_1,\ldots,x_n\in F_+$
coordinatewise by $a_{j_i}(F)$ and observe that the displayed lattice
expression at the tuple $(x_i)$ equals $x_i$.

For $F\in\mathscr L_p(E)$, we claim that $F$ has at least $n$ atoms
if and only if, for some
$\mathbf j$, all $b_F(v_i^{\mathbf j}(F))$ are positive.
If $F$ has $n$ distinct normalised atoms, they are disjoint. Approximate
them by one of these tuples and apply~\eqref{eq:atom-defect-lower} and the
Lipschitz estimate. Conversely, positivity of these $n$ defects implies
that each of the $n$ disjoint generated bands contains an atom, giving
$n$ distinct atoms. Therefore
\[
 \mathscr A_{p,\geqslant n}(E)
 =\mathscr L_p(E)\cap
   \bigcup_{\mathbf j\in\omega^n}\ \bigcup_{m\geqslant1}
   \bigcap_{i=1}^n
   \{F:b_F(v_i^{\mathbf j}(F))\geqslant1/m\}.
\]
The innermost sets are closed by upper semicontinuity, so this is $F_\sigma$.
The atomless set is the complement of
$\mathscr A_{p,\geqslant1}(E)$ relative to the closed class
$\mathscr L_p(E)$. Exactly $n$ atoms is expressed by
$\mathscr A_{p,\geqslant n}(E)\setminus
\mathscr A_{p,\geqslant n+1}(E)$, and infinitely many by
$\bigcap_{n\geqslant1}\mathscr A_{p,\geqslant n}(E)$.
These yield the stated ranks. Finally, the lattice-isometry class of
$L_p[0,1]$ consists exactly of the non-zero atomless members of
$\mathscr L_p(E)$ by~\eqref{eq:separable-lp-types}.
The zero code is closed, so this class is $G_\delta$.
\end{proof}

\begin{proposition}\label{prop:compact-order-intervals}
For every separable Banach lattice $E$, the class
\[
 \mathscr K(E)=\{F\in\Subl(E):[0,x]\text{ is norm compact for every }x\in F_+\}
\]
is $\boldsymbol\Pi^0_3$. Consequently, the lattice-isometry classes of
$\ell_p$ and $c_0$ are $\boldsymbol\Pi^0_3$ in $\Subl(E)$.
For each finite $n\geqslant1$, the lattice-isometry classes of
$\ell_p^n$ and $\ell_\infty^n$ are locally closed.
\end{proposition}

\begin{proof}
Keep the positive dense selections $a_j$. For fixed $F,i$, the sequence
$(a_j(F)\wedge a_i(F))_{j\in\omega}$ is dense in $[0,a_i(F)]$.
This interval is closed and complete, hence compact exactly when it is
totally bounded. Accordingly it is compact if and only if
\[
 \forall m\geqslant1\ \exists N\in\omega\ \forall j\in\omega:
 \min_{0\leqslant k\leqslant N}
 \|a_j(F)\wedge a_i(F)-a_k(F)\wedge a_i(F)\|\leqslant1/m.
\]
For fixed $i,m,N,j$ the displayed inequality is a closed condition.
Thus simultaneous compactness for all $i$ is $\boldsymbol\Pi^0_3$.
It also gives compactness of every $[0,x]$, $x\in F_+$.
Indeed, for $0\leqslant y\leqslant x$,
\[
 \|y-y\wedge a_i(F)\|
   =\|(y-a_i(F))^+\|\leqslant\|x-a_i(F)\|.
\]
Choose $a_i(F)$ close to $x$ and use a finite net for $[0,a_i(F)]$.
This proves total boundedness of $[0,x]$, and hence the assertion about
$\mathscr K(E)$.

The standard compact-order-interval characterisation says that a Banach
lattice belongs to $\mathscr K(E)$ precisely when it is atomic with order
continuous norm; see~\cite[Section~7]{KLT2018}.
In such a separable lattice, its normalised atoms form a finite or countable
family with dense linear span. For an abstract $L_p$ norm their finite
linear combinations have norm
$(\sum_i|\alpha_i|^p)^{1/p}$; for an AM-norm they have norm
$\max_i|\alpha_i|$. Completion therefore gives, respectively,
$\ell_p(A)$ and $c_0(A)$, with finite or countable $A$.
It follows that the infinite-dimensional lattice-isometry classes are
\begin{align*}
 [\ell_p]_{\mathrm{lat}}&=\mathscr L_p(E)\cap\mathscr K(E)\cap
       \{F:\dim F=\infty\},\\
 [c_0]_{\mathrm{lat}}&=\mathscr M(E)\cap\mathscr K(E)\cap
       \{F:\dim F=\infty\}.
\end{align*}
The set $\{F:\dim F\geqslant n\}$ is open for each finite $n$.
Indeed, linear independence of a fixed finite tuple is an open condition,
and continuous dense selectors witness every independent tuple by small
perturbation. Thus infinite dimension is $G_\delta$, giving the two
$\boldsymbol\Pi^0_3$ bounds.

Every $n$-dimensional vector lattice has a basis of $n$ disjoint positive
atoms. After normalising these atoms, the abstract $L_p$ identity gives
$\ell_p^n$ and the AM-identity gives $\ell_\infty^n$. Therefore these two
finite-dimensional classes are, respectively,
\[
 \mathscr L_p(E)\cap\{F:\dim F=n\},\qquad
 \mathscr M(E)\cap\{F:\dim F=n\}.
\]
Each is the intersection of a closed and an open set, as required.
\end{proof}

\begin{theorem}\label{thm:finite-dim-complete}
In $\Subl(\mathcal C)$ and in $\Ideal(\FBL[\ell_1])$, the class of
finite-dimensional sublattices
\[
 \mathrm{FinSubl}=\{F:\dim F<\infty\}
\]
is $\boldsymbol\Sigma^0_2$-complete.  Consequently,
the class of infinite-dimensional sublattices is $\boldsymbol\Pi^0_2$-complete.
\end{theorem}

\begin{proof}
For any separable Banach lattice $E$, Proposition~\ref{prop:compact-order-intervals}
establishes that $\{F\in\Subl(E):\dim F\geqslant n\}$ is Wijsman-open for every
$n\in\N$.  Thus $\{F:\dim F\leqslant n\}$ is closed, and
$\mathrm{FinSubl}(E)=\bigcup_{n=0}^\infty \{F:\dim F\leqslant n\}$ is
$F_\sigma$.

For the lower bound in $\Subl(\mathcal C)$, consider the standard
$\boldsymbol\Sigma^0_2$-complete set
\[
 \mathscr P_2=\{a\in\{0,1\}^\N:\exists N\in\N\ \forall n\geqslant N,\ a_n=0\}.
\]
In Lemma~\ref{lem:two-choice-realisation}, set $V_i^0=\{0\}$ and $V_i^1=\R$
for each $i\in\N$, embedded into $U=C([0,1];\ell_1^2)$ as single-dimensional
sublattices $L_i^1=\R e_i$ and $L_i^0=\{0\}$.  Then $H_i=L_i^0+L_i^1=\R$,
so $H=(\bigoplus_i H_i)_{\ell_2}=\ell_2$ is reflexive.
Lemma~\ref{lem:two-choice-realisation} yields a Wijsman-continuous map
\[
 \Phi\colon\{0,1\}^\N\longrightarrow\Subl(\mathcal C)
\]
such that
\[
 \Phi(a)\cong_{\mathrm{lat,iso}}\left(\bigoplus_{i:\,a_i=1}\R\right)_{\ell_2}.
\]
Then $\dim\Phi(a)=\sum_{i=1}^\infty a_i$.  Consequently,
\[
 a\in\mathscr P_2\quad\Longleftrightarrow\quad\dim\Phi(a)<\infty.
\]
This is a continuous reduction of $\mathscr P_2$ to $\mathrm{FinSubl}(\mathcal C)$,
proving $\boldsymbol\Sigma^0_2$-completeness.  Composing $\Phi$ with the
topological embedding $\Psi_{\mathcal C}\colon\Subl(\mathcal C)\to\Ideal(\FBL[\ell_1])$
from Theorem~\ref{thm:coding-embedding} proves completeness in the quotient coding.
\end{proof}

\begin{remark}\label{rem:classical-ranks}
All these are upper bounds in the specified coding topology; no completeness
claim is made. The purely atomic class within $\mathscr L_p(E)$ is
$\boldsymbol\Pi^0_3$, namely $\mathscr L_p(E)\cap\mathscr K(E)$.
For $n\geqslant1$, the mixed class
$L_p[0,1]\oplus_p\ell_p^n$ is consequently
$\boldsymbol\Sigma^0_3$: it is the class with exactly $n$ atoms outside
$\mathscr K(E)$. The mixed class
$L_p[0,1]\oplus_p\ell_p$ is the difference of two
$\boldsymbol\Pi^0_3$ sets, and hence is $\boldsymbol\Delta^0_4$.
These assertions concern lattice isometries. In particular, the atomless
and atomic classes remain distinct at $p=2$, even though their underlying
infinite-dimensional separable Hilbert spaces are linearly isometric.
The zero lattice belongs to $\mathscr L_p(E)$, $\mathscr M(E)$, and
$\mathscr K(E)$, but not to either infinite-dimensional isometry class.
\end{remark}

\subsection{Order continuity}\label{subsec:order-continuity}

\begin{proposition}\label{prop:order-continuity-coanalytic}
For every separable Banach lattice $E$, the class of closed sublattices whose
norm is order continuous is coanalytic in $\Subl(E)$.
\end{proposition}

\begin{proof}
We use the standard disjoint-sequence criterion: a Banach lattice has order
continuous norm if and only if every order-bounded disjoint sequence of
positive vectors converges to zero in norm; see~\cite{Scha1974Ban} and
\cite[Proposition~9.1]{KLT2018}.
Failure is therefore witnessed by $x\in F_+$, a sequence $(u_n)\in F_+^\omega$,
and $m\geqslant1$ such that
\[
 u_n\leqslant x,\qquad\|u_n\|\geqslant1/m\quad(n\in\omega),
 \qquad u_n\wedge u_k=0\quad(n\ne k).
\]
Passing to a subsequence gives such a uniform lower bound whenever the
criterion fails. For each fixed $m$ these conditions, including membership
in $F$, define a closed subset of
$\Subl(E)\times E\times E^\omega$.
Their countable union is Borel, so its projection, the complement of the
order-continuous class, is analytic.
\end{proof}

By Theorem~\ref{thm:reverse-borel-coding}, inverse images under $\Omega_P$
transfer the coanalytic bounds for order continuity, strict monotonicity,
and order smoothness to $\Ideal(P)$. This completes the non-uniform
assertions of Theorem~\ref{main:D}.

\subsection{Open questions}

The coanalytic upper bounds for strict monotonicity, order smoothness,
and order continuity leave their completeness open. The classical-space
bounds above also need not be optimal. It would be useful to determine
which of these complexities depend on the ambient lattice or on the
chosen admissible topology.

The Fatou, Nakano, strong Nakano, projectivity, and semiprojectivity
properties provide further questions for these coding spaces; see
\cite{Wick2007nak,ABRT2025Nakano,de2015free,KN2026Semi} for their structural
and lifting-theoretic background. For quotient-norm formulae over a fixed
dense set, the definability theorem in~\cite[Theorem~3.16]{Kan2026Pol}
provides a general source of Borel bounds. Determining which of these
properties admit such formulae still requires lattice-specific arguments.

\subsection*{Acknowledgements}
The author gratefully acknowledges support received from
NCN Sonata-Bis~13 (2023/50/E/ST1/00067).

\begingroup
\raggedright

\endgroup
\end{document}